\documentclass[12pt]{amsart}
\usepackage{amsmath,mathtools,amsthm,amsfonts,bigints,amssymb, mathrsfs, tikz-cd, xcolor}
\usepackage{tabularx}
\usepackage{multirow}
\newtheorem{theorem}{Theorem}[section]
\newtheorem{lemma}[theorem]{Lemma}

\newtheorem{corollary}[theorem]{Corollary}

\theoremstyle{definition}
\newtheorem{definition}[theorem]{Definition}

\newtheorem{remark}[theorem]{Remark}

\newcommand{\what}{\widehat}

\newcommand{\R}{\mathbb R}%
\newcommand{\C}{\mathbb C}%
\newcommand{\Z}{\mathbb Z}%
\newcommand{\N}{\mathbb N}%
\newcommand{\z}{\mathfrak z}%
\newcommand{\mv}{\mathfrak v}%
\newcommand{\J}{\mathscr J}%
\numberwithin{equation}{section}
\makeatletter

\renewcommand\subsubsection{\@secnumfont}{\bfseries}%
\renewcommand\subsubsection{\@startsection{subsubsection}{3}
  \z@{.5\linespacing\@plus.7\linespacing}{-.5em}%
  {\normalfont\bfseries}}

  \makeatother
\makeatletter
\@namedef{subjclassname@2020}{%
  \textup{2020} Mathematics Subject Classification}
\makeatother

\begin{document}

\title[Maximal estimates and pointwise convergence]{Maximal estimates and pointwise convergence for solutions of certain dispersive equations with radial initial data on Damek-Ricci spaces}

\author[Utsav Dewan]{Utsav Dewan}
\address{Stat-Math Unit, Indian Statistical Institute, 203 B. T. Rd., Kolkata 700108, India}
\email{utsav\_r@isical.ac.in}



\begin{abstract}
One of the most celebrated problems in Euclidean Harmonic analysis is the Carleson's problem: determining the optimal regularity of the initial condition $f$ of the
Schr\"odinger equation given by
\begin{equation*}
\begin{cases}
	 i\frac{\partial u}{\partial t} -\Delta_{\mathbb{R}^n} u=0\:,\:\:\:  (x,t) \in \mathbb{R}^n \times \mathbb{R}\:, \\
	u(0,\cdot)=f\:, \text{ on } \mathbb{R}^n \:,
	\end{cases}
\end{equation*}
in terms of the index $\beta$ such that $f$ belongs to the inhomogeneous Sobolev space $H^\beta(\mathbb{R}^n)$ , so that the solution of the Schr\"odinger operator $u$ converges pointwise to $f$, $\displaystyle\lim_{t \to 0+} u(x,t)=f(x)$, almost everywhere. In this article, we address the Carleson's problem for the fractional Schr\"odinger equation, the Boussinesq equation and the Beam equation corresponding to both the Laplace-Beltrami operator $\Delta$ and the shifted Laplace-Beltrami operator $\tilde{\Delta}$, with radial initial data on Damek-Ricci spaces, by obtaining a complete description of the local (in space) mapping properties for the corresponding local (in time) maximal functions. Consequently, we obtain the sharp bound up to the endpoint $\beta \ge 1/4$, for (almost everywhere) pointwise convergence. We also establish an abstract transference principle for dispersive equations whose corresponding multipliers have comparable oscillation and also apply it in the proof of our main result.
\end{abstract}

\maketitle
\tableofcontents

\section{Introduction}
One of the most celebrated problems in Euclidean Harmonic analysis is the Carleson's problem: determining the optimal regularity of the initial condition $f$ of the
Schr\"odinger equation given by
\begin{equation*}
\begin{cases}
	 i\frac{\partial u}{\partial t} -\Delta_{\R^n} u=0\:,\:\:\:  (x,t) \in \mathbb{R}^n \times \mathbb{R} \\
	u(0,\cdot)=f\:, \text{ on } \mathbb{R}^n \:,
	\end{cases}
\end{equation*}
in terms of the index $\beta$ such that $f$ belongs to the inhomogeneous Sobolev space $H^\beta(\mathbb{R}^n)$, so that the solution of the Schr\"odinger operator $u$ converges pointwise to $f$, 
\begin{equation*}
\displaystyle\lim_{t \to 0+} u(x,t)=f(x)\:,\:\:\text{ almost everywhere }.
\end{equation*}
Such a problem was first studied by Carleson \cite{C} and Dahlberg-Kenig \cite{DK} in dimension $1$, followed by several other experts in the field for arbitrary dimension (see \cite{Cowling, Sjolin, Vega, Bourgain, DGL, DZ} and references therein).

\medskip

Recently, the investigation on Carleson's problem for the Schr\"odinger equation in the non-Euclidean setting of rank one Riemannian symmetric spaces of noncompact type or (more generally) Damek-Ricci spaces has received considerable interest \cite{WZ,Dewan,DR,Kumar,Dewan2}. Damek-Ricci spaces are also known as Harmonic $NA$ groups. These spaces $S$ are non-unimodular, solvable extensions of Heisenberg type groups $N$, obtained by letting $A=\R^+$ act on $N$ by homogeneous dilations. The rank one Riemannian Symmetric spaces of noncompact type are the prototypical examples of (and in fact accounts for a very small subclass of the more general class of) Damek-Ricci spaces \cite[p. 643]{ADY}. For unexplained notations and terminologies, the reader is referred to section $2$.

\medskip

Let $\Delta$ be the Laplace-Beltrami operator
on $S$ corresponding to the left-invariant Riemannian metric. Its $L^2$-spectrum is the half line $(-\infty,  -Q^2/4]$, where $Q$ is the homogeneous dimension of $N$. The Schr\"odinger equation on $S$ is given by
\begin{equation} \label{schrodinger}
\begin{cases}
	 i\frac{\partial u}{\partial t} -\Delta u=0\:,\:\:\:  (x,t) \in S \times \R \:,\\
	u(0,\cdot)=f\:,\: \text{ on } S \:.
	\end{cases}
\end{equation}
To quantify the Carleson's problem on $S$, for $\beta \ge 0$, we recall the fractional $L^2$-Sobolev spaces on $S$ for the special case of radial functions \cite{APV}:
\begin{equation} \label{sobolev_space_defn}
H^\beta(S):=\left\{f \in L^2(S): {\|f\|}_{H^\beta(S)}:= {\left(\int_0^\infty {\left(\lambda^2 + \frac{Q^2}{4}\right)}^\beta {|\widehat{f}(\lambda)|}^2 {|{\bf c}(\lambda)|}^{-2} d\lambda\right)}^{1/2}< \infty\right\}.
\end{equation}
One can also talk about the fractional $L^2$-Sobolev spaces for non-radial functions, by asking suitable fractional powers of (distributional) $\Delta$ of $f$ to be in $L^2$. We note that in the case $\beta=0$, $H^0(S)$ is simply $L^2(S)$. 

\medskip

In \cite{WZ}, Wang and Zhang showed that on Real Hyperbolic spaces, $\beta > 1/2$ is a sufficient condition for the (almost everywhere) pointwise convergence of the solution of the Schr\"odinger equation (\ref{schrodinger}). In \cite{Dewan}, the author improved the above regularity condition from $\beta > 1/2$ to $\beta \ge 1/4$, while specializing to radial initial data, on the vastly general setting of Damek-Ricci spaces. Along with Ray, the author then established the sharpness of the endpoint $\beta=1/4$ in \cite{DR}. Now unlike in $\R^n$, the spectral gap of the Laplace-Beltrami operator on non-flat manifolds such as Damek-Ricci spaces, creates difficulty in extending the arguments employed for the Schr\"odinger equation \cite{Dewan, DR} to the fractional setting (for more details see \cite[point (1), subsection 6.1]{Dewan}). In this regard, Kumar and Sajjan then considered the Carleson's problem for the Fractional Schr\"odinger equation with $a>1$, on rank one Riemannian symmetric spaces of non-compact type $G/K$, with radial initial data:  
\begin{equation} \label{frac_schrodinger}
\begin{cases}
	 i\frac{\partial u}{\partial t} +{(-\Delta )}^{a/2}u=0\:,\:\:\:  (x,t) \in G/K \times \R \\
	u(0,\cdot)=f\:,\: \text{ on } G/K \:,
	\end{cases}
\end{equation}
and proved that $\beta > 1/2$ is a sufficient condition \cite{Kumar}. 

\medskip

This gives rise to a natural question on whether $\beta > 1/2$ is also necessary, in order to guarantee the pointwise convergence of the solution for the Fractional Schr\"odinger equation with radial initial data, on rank one Riemannian symmetric spaces of noncompact type or (more generally) Damek-Ricci spaces. In this article, we answer the above question in the negative and improve their bound $\beta > 1/2$ down to the sharp bound $\beta \ge 1/4$. In fact, we do much more: we obtain a complete description of the local (in space) mapping properties for the local (in time) maximal function corresponding to more general dispersive equations with radial initial data on Damek-Ricci spaces.

\medskip

The Fractional Schr\"odinger equation is a special example of dispersive equations of the form,
\begin{equation} \label{dispersive}
\begin{cases}
	 i\frac{\partial u}{\partial t} +\Psi(\sqrt{-\Delta} )u=0\:,\:  (x,t) \in S \times \R \\
	u(0,\cdot)=f\:,\: \text{ on } S \:,
	\end{cases}
\end{equation}
where $\Psi: (0,\infty) \to \R$ satisfies suitable conditions. We note that for $\Psi(r):= r^a$, with $a>1$, we recover the Fractional Schr\"odinger equation (\ref{frac_schrodinger}). Some other examples are given by $\Psi(r):= r \sqrt{1+r^2}$ and $\Psi(r):= \sqrt{1+r^4}$ corresponding to the Boussinesq equation and the Beam equation (for more details see \cite{FX, GPW}), which often appear in Mathematical Physics. Then for a radial function $f$ belonging to the $L^2$-Schwartz class $\mathscr{S}^2(S)_o$ (for the definition see (\ref{schwartz_defn})), the solution of (\ref{dispersive}) is given by,
\begin{equation} \label{dispersive_soln}
S_{\psi,t} f(x):= \int_{0}^\infty \varphi_\lambda(x)\:e^{it\psi(\lambda)}\:\widehat{f}(\lambda)\: {|{\bf c}(\lambda)|}^{-2}\: d\lambda\:,
\end{equation}
where $\varphi_\lambda$ are the spherical functions, $\psi(\lambda):=\Psi\left(\sqrt{\lambda^2 + \frac{Q^2}{4}}\right)$, is the phase function of the corresponding multiplier, $\widehat{f}$ is the Spherical Fourier transform of $f$ and ${\bf c}(\cdot)$ denotes the Harish-Chandra's ${\bf c}$-function. By replacing $\Delta$ with the shifted Laplace-Beltrami operator $\tilde{\Delta}:=\Delta + \frac{Q^2}{4}$, in (\ref{dispersive}), we would also consider the shifted variants of the Fractional Schr\"odinger equation, the Boussinesq equation and the Beam equation. We note that in this case, due to the absence of the spectral gap, one simply has $\psi(\lambda)=\Psi(\lambda)$, reminiscent of the classical Euclidean situation.   

\medskip

Let $B_R$ denote the geodesic ball on $S$, centered at the identity $e$ and radius $R>0$. As in the literature, to address the problem of pointwise convergence, the key is to consider the corresponding maximal function,
\begin{equation} \label{maximal_fn}
S^*_{\psi} f(x):= \displaystyle \sup_{0<t<1} \left|S_{\psi, t} f(x)\right|\:,
\end{equation}
and then obtain its $L^p_{loc}$ boundedness on balls, for some $p \in [1,\infty]$,
\begin{equation} \label{estimates_on_balls}
{\|S^*_{\psi} f\|}_{L^p(B_R)} \lesssim {\|f\|}_{H^\beta(S)}\:,
\end{equation}
for all $R>0$. As mentioned before, the main result of this article is the following complete description of the pairs $(p, \beta) \in [1, \infty] \times [0,\infty)$, for which one has the maximal estimates (\ref{estimates_on_balls}) for any $R>0$, and every radial $L^2$-Schwartz class function $f$ on $S$:
\begin{theorem} \label{theorem}
For the maximal functions corresponding to the Fractional Schr\"odinger equation (with $a>1$), the Boussinesq equation and the Beam equation with respect to both $\Delta$ and $\tilde{\Delta}$, we have:
\begin{itemize}
\item[(i)] If $\beta < 1/4$, then (\ref{estimates_on_balls}) holds for no $p$.
\item[(ii)] If $1/4 \le \beta < n/2$, then (\ref{estimates_on_balls}) holds if and only if  $p \le 2n/(n-2\beta)$.
\item[(iii)] If $\beta = n/2$, then (\ref{estimates_on_balls}) holds if and only if $p< \infty$.
\item[(iv)] If $\beta > n/2$, then (\ref{estimates_on_balls}) holds for all $p$.
\end{itemize}
\end{theorem}

Then by standard arguments in the literature (for instance see the proof of Theorem 5 of \cite{Sjolin}), one gets the desired result on pointwise convergence:
\begin{corollary} \label{cor}
For the Fractional Schr\"odinger equation (with $a > 1$), the Boussinesq equation and the Beam equation corresponding to both $\Delta$ and $\tilde{\Delta}$, the pointwise convergence
\begin{equation*}
\displaystyle\lim_{t \to 0+} S_{\psi,t} f(x)=f(x)\:,
\end{equation*}
holds for almost every $x$ in $S$ with respect to the left Haar measure on $S$, whenever the radial initial data $f \in H^\beta(S)$ with $\beta \ge 1/4$.
\end{corollary}

\begin{remark}
\begin{enumerate}
\item[(i)] In the special case of rank one Riemannian symmetric spaces of non-compact type, Kumar and Sajjan obtained $\beta> 1/2$ to be a sufficient condition for the pointwise convergence to hold true in the context of the Fractional Schr\"odinger equation ($a>1$) with radial initial data \cite[Theorem A]{Kumar}. Our corollary \ref{cor} improves the regularity threshold down to $\beta \ge 1/4$. Theorem \ref{theorem} also establishes the sharpness of the endpoint $\beta=1/4$.
\item[(ii)] Theorem \ref{theorem} generalizes the local mapping properties for the local maximal function corresponding to the Schr\"odinger equation \cite[Theorem 1.1]{DR} to more general class of dispersive equations with radial initial data (see subsection $6.1$).
\end{enumerate}
\end{remark}

The Carleson's problem for the fractional Schr\"odinger equation ($a>1$) with radial initial data on $\R^n$ was first studied by Prestini in \cite{Prestini}, where she proved the sharp upto the endpoint bound $\beta \ge 1/4$. Afterwards, the complete mapping properties of the corresponding maximal function were obtained by Sj\"olin in \cite{Sjolin2}, which were partially generalized to more general dispersive equations by Ding and Niu \cite{DN}. Their arguments hinge heavily on the small and large argument asymptotics of the Bessel functions and the oscillation afforded by them and the corresponding mutiplier. Unlike in the Euclidean case, in the general setting of Damek-Ricci spaces however, no explicit expression of the spherical function is known. In fact, it only admits certain series expansions (the Bessel series expansion and a series expansion similar to the classical Harish-Chandra series expansion) depending on the geodesic distance from the identity of the group.  

\medskip

We commence the proof of the sufficient conditions in Theorem \ref{theorem} by introducing two cut-offs near $0$ and at $\infty$ and then decomposing the linearized maximal function into two pieces corresponding to small and large frequency. The former is easy to handle. Estimating the latter one however is more involved, as it is the high frequency part that essentially dictates maximal estimates. In order to do so, we invoke the various series expansions of $\varphi_\lambda$. This prompts us to decompose the ball $B_R$ into a closed ball $\overline{B_{R_0}}$ and an open annulus $A(R_0,R)$,
\begin{equation*}
B_R = \overline{B_{R_0}} \sqcup A(R_0,R)\:.
\end{equation*}
On the ball $\overline{B_{R_0}}$, we expand the spherical function into the Bessel series. Then by building up on ideas employed by the author and Ray in \cite{DR}, we estimate the linearized maximal function on $\overline{B_{R_0}}$ by forming a connection with $\R^n$ and lifting Euclidean results to $S$. 

\noindent On the annulus $A(R_0,R)$, the above mentioned Bessel series expansion is no longer valid and thus this is the genuine non-Euclidean situation. We first consider the high regularity scenario, that is $\beta >1/2$. In this case, the pointwise decay of $\varphi_\lambda$ given by (\ref{pointwise_phi_lambda}) suffices to estimate the linearized maximal function. In the lower regularity regime $1/4 \le \beta \le 1/2$ however, it is much more delicate. Invoking a variation of the Harish-Chandra series expansion of $\varphi_\lambda$, recently obtained by Anker et. al. \cite{APV},
we decompose the linearized maximal function into three parts: two of them are oscillatory in nature, while the third one corresponds to certain error terms. To estimate the oscillatory parts, we first employ an abstract $TT^*$ argument. Then we expand the cut-off near $\infty$ into infinitely many small pieces. On each such small window of frequency, we estimate the joint oscillation afforded by the spherical functions and the multipliers in such a way that the estimates turn out to be summable. In this oscillatory integral estimate, we dispense of the explicitness of the phase function $\psi$ and only use large frequency asymptotics of $\psi'$ and $\psi''$. This step is somewhat technical. Summing up the above estimates, we appeal to the one dimensional Pitt's inequality to obtain the desired estimate. The third part corresponding to the error term is taken care of by certain pointwise estimates. Here one crucially uses the extra decay (compared to the classical case of the Harish-Chandra series expansion) of the $\Gamma$ coefficients in $\lambda$, given by (\ref{coefficient_estimate}).

\medskip

For the proof of sharpness of the end-point $\beta=1/4$, we first consider the case of the fractional Schr\"odinger equation corresponding to $\Delta$. In this case, we work in a small annulus and hence naturally expand the spherical function in the Bessel series. The oscillation of the spherical function is realized in terms of the leading Bessel term in the series. Then invoking the oscillatory asymptotic expansion for that Bessel function, we eventually decompose the linearized maximal function into three pieces. Controlling the joint oscillation afforded by the spherical function and the multiplier, we realize that one of the above pieces exhibits divergence while the other two decay. Our computations involve derivative estimates of the Harish-Chandra's ${\bf c}$-function and again, the spectral gap of $\Delta$ makes the computations more difficult (compared to Sj\"olin's arguments in $\R^n$). The fractional Schr\"odinger equation corresponding to $\tilde{\Delta}$ follows exactly the same line of arguments and is perhaps even a bit simpler due to the absence of the spectral gap. The other equations are then taken care of by applying an abstract transference principle (Theorem \ref{transference_principle}), which we explain next. 

\begin{definition} \label{my_defn}
Let us consider two dispersive equations of the form (\ref{dispersive}) corresponding to $\Psi_1,\Psi_2$ (recall that $\psi_1,\psi_2$ are the phase functions of the corresponding multipliers).
\begin{itemize}
\item[(i)] They are called $(\psi_1,\psi_2)$-{\it locally transferrable} if given that for all $R>0$, some $\beta_0>0$ and some $p \in [1,\infty]$, the maximal estimate 
\begin{equation*}
{\|S^*_{\psi_1} f\|}_{L^p(B_R)} \lesssim {\|f\|}_{H^\beta(S)}\:,
\end{equation*}
holds for all $\beta >\beta_0$ and all $f \in \mathscr{S}^2(S)_o$, it follows that the maximal estimate 
\begin{equation*}
{\|S^*_{\psi_2} f\|}_{L^p(B_R)} \lesssim {\|f\|}_{H^\beta(S)}\:,
\end{equation*}
also holds for all $R>0$, all $\beta >\beta_0$ and all $f \in \mathscr{S}^2(S)_o$. 

\item[(ii)] If the two equations are both $(\psi_1,\psi_2)$-locally transferrable as well as $(\psi_2,\psi_1)$-locally transferrable, then they are simply called {\it locally transferrable}.

\item[(iii)] If there exist $\Lambda>0$ and $C>0$, such that the phase functions $\psi_1$ and $\psi_2$ satisfy
\begin{equation*}
\left|\psi_1(\lambda)-\psi_2(\lambda)\right| \le C\:, \text{ for all } \lambda > \Lambda\:,
\end{equation*} 
then the equations are said to be {\it of comparable oscillation}. 
\end{itemize}
\end{definition}

\begin{remark} \label{examples_remark}
Point $(iii)$ of definition \ref{my_defn} means that in high frequency, both the phase functions are within bounded error. Let us now consider the equations mentioned before.
\begin{itemize}
\item[(i)] For $a>1$, the phase function of the multiplier of the Fractional Schr\"odinger equation corresponding to $\Delta$ and the shifted operator $\tilde{\Delta}$ are given by $\psi_1(\lambda)={\left(\lambda^2 + \frac{Q^2}{4}\right)}^{a/2}$ and $\psi_2(\lambda)=\lambda^a$ respectively. Now for $1<a \le 2$ and $\lambda$ large, 
\begin{equation*}
{\left(\lambda^2 + \frac{Q^2}{4}\right)}^{a/2} = \lambda^a + \mathcal{O}(1)\:,
\end{equation*}
and thus for $1<a \le 2$, both the Fractional Schr\"odinger equations are of comparable oscillation. This is clearly not the case when $a>2$.
\item[(ii)] The phase functions of the multiplier of the Boussinesq operator and its shifted counterpart are given respectively as,
\begin{equation*}
\psi_3(\lambda)= \sqrt{\left(\lambda^2 + \frac{Q^2}{4}\right)} \sqrt{1+\left(\lambda^2 + \frac{Q^2}{4}\right)} \:,\text{ and } \psi_4(\lambda)= \lambda \sqrt{1+\lambda^2} \:.
\end{equation*}
For $\lambda$ large,
\begin{eqnarray*}
&&\psi_3(\lambda) = \left(\lambda + \mathcal{O}(\lambda^{-1})\right)\left(\lambda + \mathcal{O}(\lambda^{-1})\right)= \lambda^2 + \mathcal{O}(1)\:, \\
&&\psi_4(\lambda)= \lambda \left(\lambda + \mathcal{O}(\lambda^{-1})\right) = \lambda^2 + \mathcal{O}(1)\:,
\end{eqnarray*}
and thus both are of comparable oscillation to the classical Schr\"odinger equation (that is, $a=2$). The same is true for both the Beam operators and can be verified similarly. 
\end{itemize}
\end{remark}

The notion of {\it comparable oscillation} defines an equivalence relation. Intuitively, all the equations in the same equivalence class should `behave similarly' as essentially, it is the oscillation of the multiplier in the high frequency regime that dictates maximal estimates. Our next result confirms this.

\begin{theorem} \label{transference_principle}
Let $\psi_1$ and $\psi_2$ be continuous real-valued functions on $[0,\infty)$ such that they are $C^\infty$ away from the origin. If the dispersive equations corresponding to $\psi_1$ and $\psi_2$ are of comparable oscillation, then they are also locally transferrable.  
\end{theorem} 

This article is organized as follows. In section $2$, we fix our notations, recall certain aspects of Euclidean Fourier Analysis, first principles of oscillatory integral estimates in dimension one and the essential preliminaries about Damek-Ricci spaces and Spherical Fourier Analysis thereon. The sufficient conditions of Theorem \ref{theorem} are proved in section $3$. In section $4$, the transference principle (Theorem \ref{transference_principle}) is proved. The necessary conditions of Theorem \ref{theorem} are proved in section $5$. Finally, we conclude in section $6$, by making some remarks and posing some new problems.
 
\section{Preliminaries}
In this section, we recall some preliminaries and fix our notations.
\subsection{Some notations}
Throughout, the symbols `c' and `C' will denote positive constants whose values may change on each occurrence. The enumerated constants $C_1,C_2, \dots$ will however be fixed throughout. $\N$ will denote the set of positive integers. For $x \in \R$, $\lceil x\rceil$ will denote the smallest integer $m \in \Z$ such that $x \le m$. For non-negative functions $f_1,\:f_2,\:f_3$ we write, 
\begin{itemize}
\item $f_1 \lesssim f_2$ if there exists a positive constant $C $, so that
\begin{equation*}
f_1 \le C f_2 \:\:.
\end{equation*}
\item $f_1 \asymp f_2$ if there exist constants $C,C'>0$, so that
\begin{equation*}
C f_1 \le f_2 \le C' f_1\:.
\end{equation*}
\item $f_1=f_2+\mathcal{O}(f_3)$ if 
\begin{equation*}
|f_1-f_2| \lesssim f_3\:.
\end{equation*}
\end{itemize}
 
\subsection{Fourier Analysis on $\R^n$:}
In this subsection, we recall some Euclidean Fourier Analysis, most of which can be found in \cite{SW}. On $\R$, for ``nice" functions $f$, the Fourier transform $\tilde{f}$ is defined as
\begin{equation*}
\tilde{f}(\xi)= \int_\R f(x)\: e^{-ix\xi}\:dx\:.
\end{equation*}
We next state the Pitt's inequality for one-dimensional Fourier transforms:
\begin{lemma} \cite[p. 489]{Stein} \label{Pitt's_ineq}
One has the inequality
\begin{equation*}
{\left(\int_\R {\left|\tilde{f}(\xi)\right|}^2 \:{|\xi|}^{-2 \beta} d\xi\right)}^{1/2} \lesssim {\left(\int_\R {\left|f(x)\right|}^p\: {|x|}^{\beta_1 p} dx\right)}^{1/p}\:,
\end{equation*}
where $\beta_1 = \beta + \frac{1}{2}- \frac{1}{p}$ and the following two conditions are satisfied:
\begin{equation*}
0 \le \beta_1 < 1 - \frac{1}{p}\:,\: \text{ and } 0 \le \beta < \frac{1}{2}\:.
\end{equation*}
\end{lemma}
A $C^\infty$ function $f$ on $\R$ is called a Schwartz class function if 
\begin{equation*}
\left|{\left(\frac{d}{dx}\right)}^M f(x)\right| \lesssim {(1+|x|)}^{-N} \:, \text{ for any } M, N \in \N \cup \{0\}\:.
\end{equation*}
We denote by $\mathscr{S}(\R)$ the class of all such functions and $\mathscr{S}(\R)_{e}$ will denote the collection of all even Schwartz class functions on $\R$. Similarly, $C^\infty_c(\R)_e$ denotes the collection of all even, compactly supported smooth functions on $\R$. 

\medskip

Let $\beta \in \C$ with $Re(\beta)>0$. The Riesz potential of order $\beta$ is the operator
\begin{equation*}
I_\beta= {(-\Delta_{\R})}^{-\beta/2}\:,
\end{equation*}
which can also be written as,
\begin{equation*}
I_\beta(f)(x)=C \int_\R f(y)\: {|x-y|}^{\beta -1}\:dy\:,
\end{equation*}
for some $C>0$ (depending on $\beta$), whenever $f \in \mathscr{S}(\R)$. For $\beta >0$, the Fourier transform of the Riesz potential of $f \in \mathscr{S}(\R)$ satisfies the following identity:
\begin{equation} \label{riesz_identity}
(I_\beta(f))^\sim(\xi)= C_\beta {|\xi|}^{-\beta}\tilde{f}(\xi)\:.
\end{equation}
For a ``nice" radial function $f$ in $\R^n$, $n \ge 2$, the Euclidean spherical Fourier transform is defined as
\begin{equation*}
\mathscr{F}f(\lambda):= \int_0^\infty f(r) \J_{\frac{n-2}{2}}(\lambda r)\: r^{n-1}\: dr\:,\:\:\:\:\:\lambda\in [0,\infty),
\end{equation*}
where for all $\mu \ge 0\:, \J_\mu$ denotes the modified Bessel function:
\begin{equation*}
\J_\mu(z)= 2^\mu \: \pi^{1/2} \: \Gamma\left(\mu + \frac{1}{2} \right) \frac{J_\mu(z)}{z^\mu}\:.
\end{equation*}
Here $J_\mu$ are the Bessel functions \cite[p. 154]{SW}. The asymptotic oscillation of Bessel functions is well-studied:
\begin{lemma}\cite[lemma 1]{Prestini} \label{bessel_function_expansion}
Let $\mu \in \frac{1}{2}\N$. Then
there exists a positive constant $A_\mu$ such that
\begin{equation*}
J_\mu(s)= \sqrt{\frac{2}{\pi s}} \cos \left( s - \frac{\pi}{2}\mu - \frac{\pi}{4}\right) + \tilde{E}_\mu(s)\:,
\end{equation*}
where
\begin{equation*}
|\tilde{E}_\mu(s)| \lesssim s^{-3/2}\:,\:\text{ for } s \ge A_\mu\:.
\end{equation*}
\end{lemma}
\subsection{Preliminaries on oscillatory integral estimates:}
In this subsection, we recall some preliminary results that are useful in estimating oscillatory integrals. The first one is the standard van der Corput's lemma:
\begin{lemma} \cite[p. 309]{BigStein} \label{van_der_corput}
Let $\zeta \in C^\infty_c(\R)$ and a real-valued function $\theta \in C^\infty(\R)$ satisfying 
\begin{equation*}
|\theta''(\xi)| > \lambda >0 \:,\text{ for all } \xi \in Supp(\zeta)\:.
\end{equation*}
Then we have
\begin{equation*}
\left|\int_{\R} e^{i\theta(\xi)}\:\zeta(\xi)\:d\xi\right| \lesssim \lambda^{-1/2}\: \left\{\|\zeta\|_\infty \:+\: \|\zeta'\|_1 \right\}\:.
\end{equation*}
\end{lemma}

We will also require the following lemma which is essentially integration by parts:
\begin{lemma} \cite{SjolinOsc} \label{Sjolin_lemma}
Let $I$ denote an open interval in $\R$. Assume $\zeta \in C^\infty_c(I)$ and a real-valued function $\theta \in C^\infty(I)$ with $\theta' \ne 0$. If $l$ is a positive integer, then
\begin{equation*}
\int_I e^{i \theta(\xi)}\: \zeta(\xi)\: d\xi= \int_I e^{i \theta(\xi)}\: \kappa_l(\xi)\: d\xi\:, 
\end{equation*}
where $\kappa_l$ is a linear combination of functions of the form,
\begin{equation*}
\zeta^{(m)} {\left(\theta'\right)}^{-l-r} \displaystyle\prod_{q=1}^r \theta^{(j_q)}\:,\:\text{ for } 0 \le m,r \le l\:,\: 2 \le j_q \le l+1\:.
\end{equation*}

\end{lemma}

\subsection{Damek-Ricci spaces and spherical Fourier Analysis thereon:}
In this subsection, we will explain the notations and state relevant results on Damek-Ricci spaces. Most of these results can be found in \cite{ADY, APV, A}.

\medskip

Let $\mathfrak n$ be a two-step real nilpotent Lie algebra equipped with an inner product $\langle, \rangle$. Let $\mathfrak{z}$ be the center of $\mathfrak n$ and $\mathfrak v$ its orthogonal complement. We say that $\mathfrak n$ is an $H$-type algebra if for every $Z\in \mathfrak z$ the map $J_Z: \mathfrak v \to \mathfrak v$ defined by
\begin{equation*}
\langle J_z X, Y \rangle = \langle [X, Y], Z \rangle, \:\:\:\: X, Y \in \mathfrak v
\end{equation*}
satisfies the condition $J_Z^2 = -|Z|^2I_{\mathfrak v}$, $I_{\mathfrak v}$ being the identity operator on $\mathfrak v$. A connected and simply connected Lie group $N$ is called an $H$-type group if its Lie algebra is $H$-type. Since $\mathfrak n$ is nilpotent, the exponential map is a diffeomorphism
and hence we can parametrize the elements in $N = \exp \mathfrak n$ by $(X, Z)$, for $X\in \mathfrak v, Z\in \mathfrak z$. It follows from the Baker-Campbell-Hausdorff formula that the group law in $N$ is given by
\begin{equation*}
\left(X, Z \right) \left(X', Z' \right) = \left(X+X', Z+Z'+ \frac{1}{2} [X, X']\right), \:\:\:\: X, X'\in \mathfrak v; ~ Z, Z'\in \mathfrak z.
\end{equation*}
The group $A = \R^+$ acts on an $H$-type group $N$ by nonisotropic dilation: $(X, Z) \mapsto (\sqrt{a}X, aZ)$. Let $S = NA$ be the semidirect product of $N$ and $A$ under the above action. Thus the multiplication in $S$ is given by
\begin{equation*}
\left(X, Z, a\right)\left(X', Z', a'\right) = \left(X+\sqrt aX', Z+aZ'+ \frac{\sqrt a}{2} [X, X'], aa' \right),
\end{equation*}
for $X, X'\in \mathfrak v; ~ Z, Z'\in \mathfrak z; a, a' \in \R^+$.
Then $S$ is a solvable, connected and simply connected Lie group having Lie algebra $\mathfrak s = \mathfrak v \oplus \mathfrak z \oplus \R$ with Lie bracket
\begin{equation*}
\left[\left(X, Z, l \right), \left(X', Z', l' \right)\right] = \left(\frac{1}{2}lX' - \frac{1}{2} l'X, lZ'-lZ + [X, X'], 0\right).
\end{equation*}

We write $na = (X, Z, a)$ for the element $\exp(X + Z)a, X\in \mathfrak v, Z \in \mathfrak z, a\in A$. We note
that for any $Z \in \mathfrak z$ with $|Z| = 1$, $J_Z^2 = -I_{\mathfrak v}$; that is, $J_Z$ defines a complex structure
on $\mathfrak v$ and hence $\mathfrak v$ is even dimensional. $m_\mv$ and $m_z$ will denote the dimension of $\mv$ and $\z$ respectively. Let $n$ and $Q$ denote dimension and the homogenous dimension of $S$ respectively:
\begin{equation*}
n=m_{\mv}+m_\z+1 \:\: \textit{ and } \:\: Q = \frac{m_\mv}{2} + m_\z.
\end{equation*}

\medskip

The group $S$ is equipped with the left-invariant Riemannian metric induced by
\begin{equation*}
\langle (X,Z,l), (X',Z',l') \rangle = \langle X, X' \rangle + \langle Z, Z' \rangle + ll'
\end{equation*}
on $\mathfrak s$. For $x \in S$, we denote by $s=d(e,x)$, that is, the geodesic distance of $x$ from the identity $e$. Then the left Haar measure $dx$ of the group $S$ may be normalized so that
\begin{equation*}
dx= A(s)\:ds\:d\sigma(\omega)\:,
\end{equation*}
where $A$ is the density function given by,
\begin{equation*} 
A(s)= 2^{m_\mv + m_\z} \:{\left(\sinh (s/2)\right)}^{m_\mv + m_\z}\: {\left(\cosh (s/2)\right)}^{m_\z} \:,
\end{equation*}
and $d\sigma$ is the surface measure of the unit sphere. As we will be dealing with finite balls centered at the identity, the following local growth asymptotics of the density function will be useful:
\begin{equation}\label{density_function}
A(s) \asymp s^{n-1}\:. 
\end{equation}

\medskip

A function $f: S \to \C$ is said to be radial if, for all $x$ in $S$, $f(x)$ depends only on the geodesic distance of $x$ from the identity $e$. If $f$ is radial, then
\begin{equation*}
\int_S f(x)~dx=\int_{0}^\infty f(s)~A(s)~ds\:.
\end{equation*}

\medskip

We now recall the spherical functions on Damek-Ricci spaces. The spherical functions $\varphi_\lambda$ on $S$, for $\lambda \in \C$ are the radial eigenfunctions of the Laplace-Beltrami operator $\Delta$, satisfying the following normalization criterion
\begin{equation*}
\begin{cases}
 & \Delta \varphi_\lambda = - \left(\lambda^2 + \frac{Q^2}{4}\right) \varphi_\lambda  \\
& \varphi_\lambda(e)=1 \:.
\end{cases}
\end{equation*}
For all $\lambda \in \R$ and $x \in S$, the spherical functions satisfy
\begin{equation*}
\varphi_\lambda(x)=\varphi_\lambda(s)= \varphi_{-\lambda}(s)\:.
\end{equation*}
It also satisfies for all $\lambda \in \R$ and all $s \ge 0$:
\begin{equation} \label{phi_lambda_bound}
\left|\varphi_\lambda(s)\right| \le 1\:.
\end{equation}

\medskip

The spherical functions are crucial as they help us define the Spherical Fourier transform of a ``nice" radial function $f$ (on $S$) in the following way:
\begin{equation*}
\widehat{f}(\lambda):= \int_S f(x) \varphi_\lambda(x) dx = \int_0^\infty f(s) \varphi_\lambda(s) A(s) ds\:.
\end{equation*}
The Harish-Chandra ${\bf c}$-function is defined as
\begin{equation*}
{\bf c}(\lambda)= \frac{2^{(Q-2i\lambda)} \Gamma(2i\lambda)}{\Gamma\left(\frac{Q+2i\lambda}{2}\right)} \frac{\Gamma\left(\frac{n}{2}\right)}{\Gamma\left(\frac{m_\mv + 4i\lambda+2}{4}\right)}\:,
\end{equation*}
for all $\lambda \in \R$. We will need the following pointwise estimates (see \cite[Lemma 4.8]{RS}):
\begin{equation} \label{plancherel_measure}
{|{\bf c}(\lambda)|}^{-2} \asymp \:{|\lambda|}^2 {\left(1+|\lambda|\right)}^{n-3}\:.
\end{equation}
For $j \in \N \cup \{0\}$, we also have the
derivative estimates (\cite[Lemma 4.2]{A}):
\begin{equation}\label{c-fn_derivative_estimates}
\left|\frac{d^j}{d \lambda^j}{|{\bf c}(\lambda)|}^{-2}\right| \lesssim_j {(1+|\lambda|)}^{n-1-j} \:, \:\: \lambda \in \R.
\end{equation}

One has the following inversion formula (when valid) for radial functions:
\begin{equation*}
f(x)= C \int_{0}^\infty \widehat{f}(\lambda)\varphi_\lambda(x) {|{\bf c}(\lambda)|}^{-2}
d\lambda\:,
\end{equation*}
where $C$ depends only on $m_\mv$ and $m_\z$. Moreover, the Spherical Fourier transform extends to an isometry from the space of radial $L^2$ functions on $S$ onto $L^2\left((0,\infty),C{|{\bf c}(\lambda)|}^{-2} d\lambda\right)$. 

\medskip

The class of radial $L^2$-Schwartz class functions on $S$, denoted by $\mathscr{S}^2(S)_{o}$, is defined to be the collection of $f \in C^\infty(S)_{o}$ such that 
\begin{equation} \label{schwartz_defn}
\left|{\left(\frac{d}{ds}\right)}^M f(s)\right| \lesssim {(1+s)}^{-N} e^{-\frac{Q}{2}s}\:, \text{ for any } M, N \in \N \cup \{0\}\:,
\end{equation}
(see \cite[p. 652]{ADY}). The spherical Fourier transform defines a topological isomorphism from $\mathscr{S}^2(S)_{o}$ to ${\mathscr{S}(\R)}_{e}$\:. We also recall \cite[p. 14]{DR}
\begin{equation*}
\mathscr{S}(\R)^\infty_{e}:=\{g\in \mathscr{S}(\R)_{e}\mid 0\notin Supp({g})\}.
\end{equation*}
So, $\mathscr{S}(\R)^\infty_{e}$ is the collection of all even Schwartz class functions on $\R$ which are supported outside an interval containing $0$. Consequently, we also look at 
\begin{equation*}
\mathscr{S}(\R^n)^\infty_{o} := \mathscr{F}^{-1}\left(\mathscr{S}(\R)^\infty_{e}\right)\:, \:\:\text{ and } \mathscr{S}^2(S)^\infty_{o} := \wedge^{-1}\left(\mathscr{S}(\R)^\infty_{e}\right)\:. 
\end{equation*}
Then by repeated applications of the Abel transform and the Schwartz multiplier corresponding to the ratio of the weights of the Plancherel measures on $\R^n$ and $S$ (for details see \cite[Eqn. (4.7), pp. 14-15]{DR}), we have the following one-one correspondence:
\begin{lemma} \label{schwartz_correspondence}
For each $f \in \mathscr{S}^2(S)^\infty_{o}$, there exists a unique $g \in \mathscr{S}(\R^n)^\infty_{o}$ (and conversely) such that 
\begin{equation*}
\lambda^{n-1} \mathscr{F}g(\lambda) = {|{\bf c}(\lambda)|}^{-2} \widehat{f}(\lambda)\:,
\end{equation*}
for all $\lambda$ .
\end{lemma}

Depending on the distance from the identity of the group, one has different series expansions of the spherical functions. We first see an expansion of $\varphi_\lambda$ in terms of Bessel functions for points near the identity. But before that, let us define the following normalizing constant in terms of the Gamma functions,
\begin{equation*}
c_0 = 2^{m_\z}\: \pi^{-1/2}\: \frac{\Gamma(n/2)}{\Gamma((n-1)/2)}\:,
\end{equation*}

\begin{lemma}\cite[Theorem 3.1]{A} \label{bessel_series_expansion}
There exist $R_0, 2<R_0<2R_1$, such that for any $0 \le s \le R_0$, and any integer $M \ge 0$, and all $\lambda \in \R$, we have
\begin{equation*}
\varphi_\lambda(s)= c_0 {\left(\frac{s^{n-1}}{A(s)}\right)}^{1/2} \displaystyle\sum_{l=0}^M a_l(s)\J_{\frac{n-2}{2}+l}(\lambda s) s^{2l} + E_{M+1}(\lambda,s)\:,
\end{equation*}
where
\begin{equation*}
a_0 \equiv 1\:,\: |a_l(s)| \le C {(4R_1)}^{-l}\:,
\end{equation*}
and the error term has the following behaviour
\begin{equation*}
	\left|E_{M+1}(\lambda,s) \right| \le C_M \begin{cases}
	 s^{2(M+1)}  & \text{ if  }\: |\lambda s| \le 1 \\
	s^{2(M+1)} {|\lambda s|}^{-\left(\frac{n-1}{2} + M +1\right)} &\text{ if  }\: |\lambda s| > 1 \:.
	\end{cases}
\end{equation*}
Moreover, for every $0 \le s <2$, the series
\begin{equation*}
\varphi_\lambda(s)= c_0 {\left(\frac{s^{n-1}}{A(s)}\right)}^{1/2} \displaystyle\sum_{l=0}^\infty a_l(s)\J_{\frac{n-2}{2}+l}(\lambda s) s^{2l}\:,
\end{equation*}
is absolutely convergent.
\end{lemma}

For the asymptotic behaviour of the spherical functions away from the identity, we look at the following series expansion \cite[pp. 735-736]{APV}:
\begin{equation} \label{anker_series_expansion}
\varphi_\lambda(s)= 2^{-m_\z/2} {A(s)}^{-1/2} \left\{{\bf c}(\lambda)  \displaystyle\sum_{\mu=0}^\infty \Gamma_\mu(\lambda) e^{(i \lambda-\mu) s} + {\bf c}(-\lambda) \displaystyle\sum_{\mu=0}^\infty \Gamma_\mu(-\lambda) e^{-(i\lambda + \mu) s}\right\}\:.
\end{equation}
The above series converges for $\lambda \in \R \setminus \{0\}$, uniformly on compacts not containing the group identity, where $\Gamma_0 \equiv 1$ and for $\mu \in \N$, one has the recursion formula,
\begin{equation*}
(\mu^2-2i\mu\lambda) \Gamma_\mu(\lambda) = \displaystyle\sum_{j=0}^{\mu -1}\omega_{\mu -j}\Gamma_{j}(\lambda)\:.
\end{equation*}
Then one has the following estimate on the coefficients \cite[Lemma 1]{APV}, for constants $C>0, d \ge 0$:
\begin{equation} \label{coefficient_estimate}
\left|\Gamma_\mu(\lambda)\right| \le C \mu^d {\left(1+|\lambda|\right)}^{-1}\:,
\end{equation}
for all $\lambda \in \R \setminus \{0\}, \mu \in \N$. Consequently, for any positive constants $c_1$ and $c_2$, one has the following pointwise decay of $\varphi_\lambda$ (\cite[Proposition 3.1]{Bhowmik}):
\begin{equation} \label{pointwise_phi_lambda}
|\varphi_\lambda(s)| \lesssim \lambda^{-\frac{n-1}{2}} e^{-\frac{Q}{2}s},\:\:\:\:\:\:\lambda>c_1,\:\:s>c_2.
\end{equation}

\medskip

The relevant preliminaries on series expansions of $\varphi_\lambda$, for the degenerate case of the Real hyperbolic spaces  can be found in \cite{ST, AP}. 

\section{Proof of the sufficient conditions of Theorem \ref{theorem}}
To study the boundedness results for the maximal function (\ref{maximal_fn}), it suffices to consider its linearization,
\begin{equation*}
T_\psi f(s): = S_{\psi,t(s)} f(s):= \int_{0}^\infty \varphi_\lambda(s)\:e^{it(s)\psi(\lambda)}\:\widehat{f}(\lambda)\: {|{\bf c}(\lambda)|}^{-2}\: d\lambda\:,
\end{equation*}
where $t(\cdot): S \to (0,1)$ is a radial measurable function and obtain inequalities of the form,
\begin{equation*} 
{\|T_\psi f\|}_{L^p(B_R)} \lesssim {\|f\|}_{H^\beta(S)}\:,
\end{equation*}
for all $f \in \mathscr{S}^2(S)_o$. 

\medskip

We commence the proof by first decomposing $T_\psi$ in terms of small and large frequency. More precisely, we consider a non-negative, even function $\eta \in C^\infty_c(\R)$ such that $Supp(\eta) \subset \left\{\xi : 1/2 < |\xi| <2\right\}$ and 
\begin{equation*}
\displaystyle\sum_{k=-\infty}^\infty \eta \left(2^{-k} \xi\right)=1\:,\: \text{ for } \xi \ne 0\:.
\end{equation*}
Let us now define,
\begin{equation*}
\eta_1(\xi):= \displaystyle\sum_{k=-\infty}^0 \eta(2^{-k} \xi) \:,\text{   and  } \eta_2(\xi):= \displaystyle\sum_{k=1}^\infty \eta(2^{-k} \xi) \:.
\end{equation*}
We note that both $\eta_1$ and $\eta_2$ are even non-negative smooth functions with $Supp(\eta_1) \subset (-2,2)$, $Supp(\eta_2) \subset \R \setminus (-1,1)$ and $\eta_1+\eta_2 \equiv 1$. Accordingly, we decompose $T_\psi$ as,
\begin{equation}\label{first_decomposition}
T_\psi f(s): = T_{\psi,1} f(s) + T_{\psi,2} f(s)\:,
\end{equation}
where,
\begin{eqnarray*}
T_{\psi,1} f(s):= \int_{0}^2 \varphi_\lambda(s)\:e^{it(s)\psi(\lambda)}\:\widehat{f}(\lambda)\:\eta_1(\lambda)\: {|{\bf c}(\lambda)|}^{-2}\: d\lambda\:, \\
T_{\psi,2} f(s):= \int_{1}^\infty \varphi_\lambda(s)\:e^{it(s)\psi(\lambda)}\:\widehat{f}(\lambda)\:\eta_2(\lambda)\: {|{\bf c}(\lambda)|}^{-2}\: d\lambda\:.
\end{eqnarray*}
Note that $T_{\psi,1}$ and $T_{\psi,2}$ correspond to the small and large frequency respectively. We now obtain boundedness results for them seprately.

\subsection{Estimating $T_{\psi,1}$:} From the viewpoint of boundedness, $T_{\psi,1}$ is extremely well-behaved. In fact, we have:
\begin{lemma} \label{small_freq}
The inequality,
\begin{equation*}
{\|T_{\psi,1} f\|}_{L^p(B_R)} \lesssim {\|f\|}_{H^\beta(S)}\:,
\end{equation*}
holds for all $\beta \in [0,\infty)$ and all $p \in [1,\infty]$.
\end{lemma}
\begin{proof}
We note that it suffices to prove the estimate at the end-points, $p=\infty$ and $\beta=0$. In this proof, we only use the boundedness of the spherical functions and the multiplier, that is, oscillation plays no role here. Indeed, by the Cauchy-Schwarz inequality and the small frequency asymptotics of the Plancherel measure (\ref{plancherel_measure}), for $0 \le s < R$,
\begin{eqnarray*}
|T_{\psi,1} f(s)| & \le & \int_{0}^2 \left|\widehat{f}(\lambda)\right|\: {|{\bf c}(\lambda)|}^{-2}\: d\lambda \\
& \le & {\left(\int_0^\infty {|\widehat{f}(\lambda)|}^2\: {|{\bf c}(\lambda)|}^{-2} \:d\lambda\right)}^{1/2} {\left(\int_0^2  {|{\bf c}(\lambda)|}^{-2} d\lambda\right)}^{1/2} \\
&\lesssim &  {\|f\|}_{H^0(S)}\:.
\end{eqnarray*}
This proves the desired $H^0(S) \to L^\infty(B_R)$ estimate. 
\end{proof}

\subsection{Estimating $T_{\psi,2}$:} In this subsection, we are interested in the inequalities of the form,
\begin{equation} \label{maximal_estimate} 
{\|T_{\psi,2} f\|}_{L^p(B_R)} \lesssim {\|f\|}_{H^\beta(S)}\:.
\end{equation}
The boundedness properties of $T_{\psi,2}$ are summarized as follows:
\begin{lemma} \label{large_freq}
\begin{itemize} 
\item[(i)] If $1/4 \le \beta < n/2$, then (\ref{maximal_estimate}) holds if   $1\leq p \le 2n/(n-2\beta)$.
\item[(ii)] If $\beta= n/2$, then (\ref{maximal_estimate}) holds if $p \in [1, \infty)$.
\item[(iii)] If $\beta > n/2$, then (\ref{maximal_estimate}) holds for all $p \in [1,\infty]$.
\end{itemize}
\end{lemma} 

\begin{remark}
In view of the decomposition given in (\ref{first_decomposition}) and the boundedness properties of $T_{\psi,1}$ (Lemma \ref{small_freq}), to complete the proof of the sufficient conditions of Theorem \ref{theorem}, it suffices to prove Lemma \ref{large_freq}.
\end{remark}

The key strategy in the proof of Lemma \ref{large_freq} is to invoke the various series expansions of $\varphi_\lambda$. Now as the validity of these series expansions depend on the geodesic distance from the group identity, it prompts us to decompose the ball $B_R$ into a closed ball and an open annulus, $$B_R = \overline{B_{R_0}} \sqcup A(R_0,R)\:,$$
where
\begin{equation*}
\overline{B_{R_0}}=\{x \in S : d(e,x)\leq R_0 \}\:\text{ and } A(R_0,R) = \{x \in S : R_0 < d(e,x) < R\}\:,
\end{equation*}
and then to estimate $T_{\psi,2}$ on $\overline{B_{R_0}}$ and $A(R_0,R)$ separately.

\subsubsection{Estimating $T_{\psi,2}$ on $\overline{B_{R_0}}$\::}
Here, we focus on inequalities of the form:
\begin{equation} \label{maximal_estimate_ball}
{\|T_{\psi,2}f\|}_{L^p\left(\overline{B_{R_0}}\right)} \lesssim {\|f\|}_{H^{\beta}(S)}\:.
\end{equation}
\begin{lemma} \label{large_freq_ball}
\begin{itemize}
\item[(i)] If $1/4 \le \beta < n/2$, then (\ref{maximal_estimate_ball}) holds if   $1 \le p \le 2n/(n-2\beta)$.
\item[(ii)] If $\beta = n/2$, then (\ref{maximal_estimate_ball}) holds if $p \in [1,\infty)$.
\item[(iii)] If $\beta > n/2$, then (\ref{maximal_estimate_ball}) holds for all $p \in [1,\infty]$.
\end{itemize}
\end{lemma}

To prove Lemma \ref{large_freq_ball}, we make use of the Bessel series expansion of $\varphi_\lambda$ (Lemma \ref{bessel_series_expansion}). For $s\in [0,R_0]$, putting $M=0$ in Lemma \ref{bessel_series_expansion}, we get
\begin{equation} \label{ball_pf_eq1}
\varphi_\lambda(s)= c_0 {\left(\frac{s^{n-1}}{A(s)}\right)}^{1/2} \J_{\frac{n-2}{2}}(\lambda s) + E_1(\lambda,s),\:\:\:\:\:\:\:\:\:\:\lambda \ge 1\:,
\end{equation}
where
\begin{equation} \label{ball_pf_eq2}
	\left|E_1(\lambda,s) \right| \lesssim  \begin{cases}
	 s^2,  & \text{ if  }\: \lambda s \le 1 \\
	s^2 {(\lambda s)}^{-\left(\frac{n+1}{2} \right)}, &\text{ if  }\: \lambda s > 1 \:.
	\end{cases}
\end{equation}
Then invoking the decomposition of $\varphi_\lambda$ given in (\ref{ball_pf_eq1}), we further decompose $T_{\psi,2}f$ as follows
\begin{equation} \label{second_decomposition}
T_{\psi,2}f(s)=T_{\psi,3}f(s)+T_{\psi,4}f(s)
\end{equation}
where
\begin{eqnarray*}
T_{\psi,3}f(s) &=& c_0 {\left(\frac{s^{n-1}}{A(s)}\right)}^{1/2} \int_{1}^\infty \J_{\frac{n-2}{2}}(\lambda s)\:e^{it(s)\psi(\lambda)}\:\widehat{f}(\lambda)\:\eta_2(\lambda)\: {|{\bf c}(\lambda)|}^{-2}\: d\lambda\:, \\
T_{\psi,4}f(s)&=& \int_{1}^\infty E_1(\lambda, s)\:e^{it(s)\psi(\lambda)}\:\widehat{f}(\lambda)\:\eta_2(\lambda)\: {|{\bf c}(\lambda)|}^{-2}\: d\lambda\:. 
\end{eqnarray*}
Now it suffices to look at inequalities of the form,
\begin{equation} \label{maximal_estimate_ball_pieces}
{\|T_{\psi,j}f\|}_{L^p\left(\overline{B_{R_0}}\right)} \lesssim {\|f\|}_{H^{\beta}(S)},\:\:\:\text{ for }j=3,4.
\end{equation}
Hence in view of the decomposition given in (\ref{second_decomposition}), Lemma \ref{large_freq_ball} follows from the following lemmata:
\begin{lemma} \label{large_freq_ball_piece1}
 For $T_{\psi,3}$ we have,
\begin{itemize}
\item[(i)] If $1/4 \le \beta < n/2$, then (\ref{maximal_estimate_ball_pieces}) holds if   $1 \le p \le 2n/(n-2\beta)$.
\item[(ii)] If $\beta = n/2$, then (\ref{maximal_estimate_ball_pieces}) holds if $p \in [1,\infty)$.
\item[(iii)] If $\beta > n/2$, then (\ref{maximal_estimate_ball_pieces}) holds for all $p \in [1,\infty]$.
\end{itemize}
\end{lemma}

\begin{lemma} \label{large_freq_ball_piece2}
For $T_{\psi,4}$ we have,
\begin{itemize}
\item[(i)] If $n \le 4$, then (\ref{maximal_estimate_ball_pieces}) holds for all $\beta \in [0,\infty)$ and all $p \in [1,\infty]$.
\item[(ii)] If $n \ge 5$, then
\begin{itemize}
\item[(a)] If $\beta < \frac{n}{2}-2$, then (\ref{maximal_estimate_ball_pieces}) holds if $1 \le p < 2n/(n-2\beta - 4)$.
\item[(b)] If $\beta \ge \frac{n}{2}-2$, then (\ref{maximal_estimate_ball_pieces}) holds for all $p \in [1,\infty]$.
\end{itemize}
\end{itemize}
\end{lemma}
\begin{remark}
As $2n/(n-2\beta - 4)$ is larger than $2n/(n-2\beta )$, Lemmata \ref{large_freq_ball_piece1} and \ref{large_freq_ball_piece2} show that the operator $T_{\psi,4}$ is better behaved compared to $T_{\psi,3}$. 
\end{remark}
Lemma \ref{large_freq_ball_piece2} follows by using the boundedness of the multiplier and the pointwise decay of the error term $E_1$ given in (\ref{ball_pf_eq2}) and is done in \cite[Lemma 4.5, p.16]{DR}. We now focus on Lemma \ref{large_freq_ball_piece1}.
\begin{proof}[Proof of Lemma \ref{large_freq_ball_piece1}]
As $Supp(\widehat{f}\cdot \eta_2) \subset [1,\infty)$, by the Schwartz isomorphism theorem and Lemma \ref{schwartz_correspondence}, there exist $F \in \mathscr{S}^2(S)_o$ and $g \in \mathscr{S}(\R^n)_o$  such that
\begin{equation} \label{piece1_eq1}
\widehat{f}(\lambda) \eta_2(\lambda) = \widehat{F}(\lambda)\:,
\end{equation}
whereas $\mathscr{F}g$ and $\widehat{F}$ are related as follows,
\begin{equation} \label{piece1_eq2}
\lambda^{n-1}\:\mathscr{F}g(\lambda)=\:{|{\bf c}(\lambda)|}^{-2}\:\widehat{F}(\lambda) \:.
\end{equation}
Plugging the relation (\ref{piece1_eq2}) in the definition of the Sobolev norms and then using the large frequency asymptotics of ${|{\bf c}(\cdot)|}^{-2}$ given in (\ref{plancherel_measure}), it follows that the Sobolev norms of $g$ and $F$ are comparable, that is for all $\beta \ge 0$,
\begin{equation} \label{piece1_eq3}
\|g\|_{H^\beta(\R^n)} \asymp \|F\|_{H^\beta(S)}\:.
\end{equation} 

We next note that there exist real numbers $\delta_1>0$ and $\delta_2>1$ such that the first and the second derivatives of the phase function $\psi$, for the equations we are interested in, satisfy the following asymptotics:
\begin{equation} \label{phase_fn_properties}
\begin{cases}
	 |\psi'(\lambda)| \lesssim \lambda^{\delta_1 -1}\:,\text{  for } \lambda \in (0,1)  \:,\\
	 |\psi'(\lambda)| \lesssim \lambda^{\delta_2 -1}\:,\text{  for } \lambda \ge 1  \:,\\
	|\psi''(\lambda)| \asymp \lambda^{\delta_2 -2}\:,\text{  for } \lambda \ge 1  \:,
	\end{cases}
\end{equation}
with the precise case by case choices for $\delta_1$ and $\delta_2$ given in Table \ref{table:1}.


\begin{table}[h!]
\centering
\begin{tabular}{ |p{2cm}|p{7cm}|p{1cm}|p{1cm}|  }
 \hline
 Dispersive equation & Phase function of the multiplier $\psi(\lambda)$ & $\delta_1$ & $\delta_2$\\
 \hline
Fractional Schr\"odinger   &  ${\left(\lambda^2 + \frac{Q^2}{4}\right)}^{a/2}$   & $2$ &   $a$\\
\hline
Fractional Schr\"odinger (shifted) &  \:\:$\lambda^a$   & $a$ &   $a$\\
\hline
Boussinesq & ${\left(\lambda^2 + \frac{Q^2}{4}\right)}^{1/2}{\left(\lambda^2 + \frac{Q^2}{4}+1\right)}^{1/2}$ & $2$ &  $2$\\
\hline
Boussinesq (shifted) & $\lambda {\left(\lambda^2 + 1\right)}^{1/2}$ & $1$ &  $2$\\
\hline
Beam & ${\left\{1+\left(\lambda^2 + \frac{Q^2}{4}\right)^2\right\}}^{1/2}$ & $2$ &  $2$\\
\hline
Beam (shifted) & ${\left(\lambda^4 + 1\right)}^{1/2}$ & $4$ &  $2$\\
\hline
\end{tabular}

\medskip

\caption{}
\label{table:1}
\end{table}

\medskip

We now consider the corresponding Dispersive equations on $\R^n$ with radial initial data:
\begin{equation*} 
\begin{cases}
	 i\frac{\partial u}{\partial t} +\psi(\sqrt{-\Delta_{\R^n}} )u=0\:,\:\:\:  (x,t) \in \R^n \times \R \\
	u(0,\cdot)=h\:,\: \text{ on } \R^n \:,
	\end{cases}
\end{equation*}
and the Euclidean counter-part $T_{\psi,0}$ of the operator $T_{\psi}$, defined on $\mathscr{S}(\R^n)_o$, for $s \in [0,\infty)$  by,
\begin{equation*}
T_{\psi,0}\:h(s)= \int_{0}^\infty \J_{\frac{n-2}{2}}(\lambda s)\:e^{it(s)\psi(\lambda)}\:\mathscr{F}h(\lambda)\: \lambda^{n-1}\: d\lambda,\:\text{ for } h \in  \mathscr{S}(\R^n)_o\:,
\end{equation*}
where $t(\cdot)$ is the same measurable function on $[0,\infty)$. We now look at the $H^\beta(\R^n) \to L^p(B(o,R))$ boundedness properties of $T_{\psi,0}$, that is, inequalities of the form:
\begin{equation} \label{Euclidean_inequality}
{\|T_{\psi,0}\:h\|}_{L^p(B(o,R))} \lesssim {\|h\|}_{H^{\beta}(\R^n)}\:\:.
\end{equation} 
By properties of the phase functions (\ref{phase_fn_properties}), applying \cite[Theorem 1]{DN}, one has the following description:
\begin{equation} \label{Euclidean_mapping}
\begin{cases}
\text{If } 1/4 \le \beta < n/2, \text{ then } (\ref{Euclidean_inequality}) \text{ holds if }   1\leq p \le 2n/(n-2\beta) \:,\\
\text{If } \beta= n/2, \text{ then } (\ref{Euclidean_inequality}) \text{ holds if } p \in [1, \infty) \:,\\
\text{If } \beta > n/2, \text{ then } (\ref{Euclidean_inequality}) \text{ holds for all } p \in [1,\infty]\:.
\end{cases}
\end{equation}

Next using the local growth asymptotics of the density function (\ref{density_function}) for $s \in [0,R_0]$, we obtain a pointwise comparison of the moduli of $T_{\psi,3} f$ and $T_{\psi,0}\:g$ where $f$ and $g$ are related via (\ref{piece1_eq1}) and (\ref{piece1_eq2}),
\begin{eqnarray} \label{linearized_max_fn_comparison}
|T_{\psi,3} f(s)| & \asymp & \left|\int_{1}^\infty \J_{\frac{n-2}{2}}(\lambda s)\:e^{it(s)\psi(\lambda)}\:\widehat{f}(\lambda)\:\eta_2(\lambda)\: {|{\bf c}(\lambda)|}^{-2}\: d\lambda\right|\nonumber\\
&=& \left|\int_{1}^\infty \J_{\frac{n-2}{2}}(\lambda s)\:e^{it(s)\psi(\lambda)}\:\widehat{F}(\lambda)\: {|{\bf c}(\lambda)|}^{-2}\: d\lambda\right|\nonumber\\
&=& \left|\int_{1}^\infty \J_{\frac{n-2}{2}}(\lambda s)\:e^{it(s)\psi(\lambda)}\:\mathscr{F}g(\lambda)\: \lambda^{n-1}\: d\lambda\right|\nonumber \\
& = & |T_{\psi,0}\:g(s)|\:.
\end{eqnarray}
Finally, for any admissible pair $(p,\beta) \in [1,\infty] \times [0,\infty)$ appearing in the hypothesis of Lemma \ref{large_freq_ball_piece1}, we have by applying the growth asymptotics of the density function (\ref{density_function}) and the data given in (\ref{piece1_eq1})-(\ref{piece1_eq3}) and (\ref{Euclidean_mapping})-(\ref{linearized_max_fn_comparison}),
\begin{equation*}
\|T_{\psi,3} f\|_{L^p(\overline{B_{R_0}})} \lesssim \|T_{\psi,0}\:g\|_{L^p(\overline{B(o,R_0)})} \lesssim {\|g\|}_{H^{\beta}(\R^n)} \lesssim {\|F\|}_{H^{\beta}(S)} \le {\|f\|}_{H^{\beta}(S)}\:. 
\end{equation*}
This completes the proof of Lemma \ref{large_freq_ball_piece1}.
\end{proof}

\begin{remark}
Lemma \ref{large_freq_ball_piece1} is proved by forming an Euclidean connection and then lifting results on $\R^n$ to $S$. The role of the oscillation afforded by the Bessel functions and the multipliers is dormant in our arguments, as it plays a crucial role in the proof of the Euclidean result \cite[Theorem 1]{DN} itself. 
\end{remark}

\subsubsection{Estimating $T_{\psi,2}$ on $A(R_0,R)$\::}
Here, we focus on inequalities of the form:
\begin{equation} \label{maximal_estimate_annulus}
{\|T_{\psi,2}f\|}_{L^p\left(A(R_0,R)\right)} \lesssim {\|f\|}_{H^{\beta}(S)}\:.
\end{equation}
It will become gradually clear that from the viewpoint of boundedness, the behavior of  $T_{\psi,2}$ is better on the annulus $A(R_0,R)$ compared to, on the ball $\overline{B_{R_0}}$. The first instance is the higher regularity scenario, more precisely, when $\beta> 1/2$:  
\begin{lemma}\label{triviality}
If $\beta > 1/2$, then (\ref{maximal_estimate_annulus}) holds for all $p \in [1,\infty]$.
\end{lemma}
\begin{proof}
It suffices to prove the lemma at the end-point, $p=\infty$. The proof only uses the boundedness of the multiplier and the pointwise decay of $\varphi_\lambda$ given in (\ref{pointwise_phi_lambda}). Indeed, for $s\in (R_0,R)$, an application of the Cauchy-Schwarz inequality yields
\begin{eqnarray*}
|T_{\psi,2}f(s)| & \lesssim & e^{-\frac{Q}{2}s} \int_{1}^\infty \lambda^{-\left(\frac{n-1}{2}\right)}\: \left|\widehat{f}(\lambda)\right|\: {|{\bf c}(\lambda)|}^{-2}\: d\lambda \\
& \le & e^{-\frac{Q R_0}{2}} {\left(\int_1^\infty {\left(\lambda^2 + \frac{Q^2}{4}\right)}^\beta {|\widehat{f}(\lambda)|}^2 {|{\bf c}(\lambda)|}^{-2} d\lambda\right)}^{1/2} {\left(\int_1^{\infty}\frac{{|{\bf c}(\lambda)|}^{-2}\: d\lambda}{{\left(\lambda^2 + \frac{Q^2}{4}\right)}^\beta \lambda^{n-1}} \right)}^{1/2} \\
& \lesssim & {\|f\|}_{H^\beta(S)}\: {\left(\int_1^{\infty}\frac{ d\lambda}{\lambda^{2 \beta}}\right)}^{1/2}\:.
\end{eqnarray*}
As $\beta > 1/2$, the last integral converges and this completes the proof.
\end{proof}

We now focus on the lower regularity case, that is, when  $1/4 \le \beta \le 1/2$:
\begin{lemma} \label{large_freq_lemma_annulus}
If $1/4 \le \beta \le 1/2$, then (\ref{maximal_estimate_annulus}) holds if $p \in [1,4]$.
\end{lemma}
\begin{remark}
In view of Lemma \ref{large_freq_lemma_annulus}, we note that for the range $1/4 \le \beta \le 1/2$, the largest value of $p$ in Lemma \ref{large_freq_ball} is given by $2n/(n-1)$, which satisfies 
\begin{equation*}
\frac{2n}{n-1} \le 4\:,\: \text{ as } n \ge 2\:.
\end{equation*}
Thus the boundedness of $T_{\psi,2}$ on $B_R$ is essentially determined by Lemma \ref{large_freq_ball}, that is, on the small ball $\overline{B_{R_0}}$.
\end{remark}

To prove Lemma \ref{large_freq_lemma_annulus}, we invoke the  series expansion (\ref{anker_series_expansion}) of $\varphi_\lambda$. For $R_0 < s < R$ and $\lambda \ge 1$, using the series expansion  (\ref{anker_series_expansion}) and the estimate (\ref{coefficient_estimate}) on the coefficients $\Gamma_\mu$, we get
\begin{equation} \label{annulus_pf_eq1}
\varphi_\lambda (s) = 2^{-m_\z/2} {A(s)}^{-1/2} \left\{{\bf c}(\lambda)  e^{i \lambda s} + {\bf c}(-\lambda) e^{-i\lambda  s}\right\} + E_2(\lambda,s)\:,
\end{equation}
where
\begin{equation*}
E_2(\lambda,s) = 2^{-m_\z/2} {A(s)}^{-1/2} \left\{{\bf c}(\lambda)  \displaystyle\sum_{\mu=1}^\infty \Gamma_\mu(\lambda) e^{(i \lambda-\mu) s} + {\bf c}(-\lambda) \displaystyle\sum_{\mu=1}^\infty \Gamma_\mu(-\lambda) e^{-(i\lambda + \mu) s}\right\} \:,
\end{equation*}
and thus
\begin{equation} \label{annulus_pf_eq2}
\left|E_2(\lambda,s)\right| \lesssim {A(s)}^{-1/2} \left|{\bf c}(\lambda)\right| {(1+\lambda)}^{-1} \:.
\end{equation}
Hence for $R_0 < s < R$ and $\lambda \ge 1$, invoking the decomposition (\ref{annulus_pf_eq1}), we decompose $T_{\psi,2}$ as,
\begin{equation*}
T_{\psi,2}f(s)=T_{\psi,5}f(s)+T_{\psi,6}f(s)+T_{\psi,7}f(s),
\end{equation*}
where
\begin{eqnarray*}
T_{\psi,5}f(s)&=& 2^{-m_\z/2} {A(s)}^{-1/2} \bigintssss_{1}^\infty  {\bf c}(\lambda) \: e^{i\left\{\lambda s + t(s)\psi(\lambda)\right\}} \:\widehat{f}(\lambda)\: \eta_2(\lambda)\: {|{\bf c}(\lambda)|}^{-2}\: d\lambda\:, \\
T_{\psi,6}f(s)&=& 2^{-m_\z/2} {A(s)}^{-1/2} \bigintssss_{1}^\infty  {\bf c}(-\lambda) \: e^{i\left\{-\lambda s + t(s)\psi(\lambda)\right\}} \:\widehat{f}(\lambda)\: \eta_2(\lambda)\: {|{\bf c}(\lambda)|}^{-2}\: d\lambda\:, \\
T_{\psi,7}f(s)&=& \bigintssss_1^\infty E_2(\lambda,s)\: e^{it(s)\psi(\lambda)} \:\widehat{f}(\lambda)\: \eta_2(\lambda)\: {|{\bf c}(\lambda)|}^{-2}\: d\lambda\:.
\end{eqnarray*}
We now note that to prove Lemma \ref{large_freq_lemma_annulus}, it suffices to prove the following inequalities at the end-points:
\begin{lemma} \label{large_freq_lemma_annulus_pieces}
\begin{itemize}
\item[(a)] For $j=5,6$, we have 
\begin{equation*}
{\|T_{\psi,j}f\|}_{L^4\left(A(R_0,R)\right)} \lesssim {\|f\|}_{H^{1/4}(S)}\:.
\end{equation*}
\item[(b)] For $T_{\psi,7}$, we have
\begin{equation*}
{\|T_{\psi,7}f\|}_{L^\infty\left(A(R_0,R)\right)} \lesssim {\|f\|}_{H^0(S)}\:.
\end{equation*}
\end{itemize}
\end{lemma}
\begin{proof}
We first prove part $(a)$. The arguments for $T_{\psi,5}$ and $T_{\psi,6}$ being exactly similar, we focus only on $T_{\psi,5}$. The crux of the matter here is to utilize the oscillation afforded by both the spherical functions and the mutiplier. We need to prove the inequality
\begin{equation} \label{annulus_pf_eq3}
{\left(\int_{R_0}^R {\left|T_{\psi,5}f(s)\right|}^4\:A(s)\:ds\right)}^{1/4} \lesssim {\left(\int_0^\infty {\left(\lambda^2 + \frac{Q^2}{4}\right)}^{1/4} {|\widehat{f}(\lambda)|}^2\: {|{\bf c}(\lambda)|}^{-2} \:d\lambda\right)}^{1/2}\:.
\end{equation}
From the definition of $T_{\psi,5}$, it follows that
\begin{eqnarray*}
&&T_{\psi,5}f(s) A(s)^{1/4}\\
&=&  2^{-m_\z/2} {A(s)}^{-1/4} \bigintssss_{1}^\infty  {\bf c}(\lambda) \: e^{i\left\{\lambda s + t(s)\psi(\lambda)\right\}} \:\widehat{f}(\lambda)\: \eta_2(\lambda)\: {|{\bf c}(\lambda)|}^{-2}\: d\lambda \\
&=&  2^{-m_\z/2} {A(s)}^{-1/4} \bigintssss_{1}^\infty  {\bf c}(\lambda) \: e^{i\left\{\lambda s + t(s)\psi(\lambda)\right\}} \:g(\lambda)\: \eta_2(\lambda)\: {\left(\lambda^2 + \frac{Q^2}{4}\right)}^{-1/8} \:{|{\bf c}(\lambda)|}^{-1}\: d\lambda \:,
\end{eqnarray*}
where
\begin{equation*}
g(\lambda)= \widehat{f}(\lambda)\: {\left(\lambda^2 + \frac{Q^2}{4}\right)}^{1/8}\:{|{\bf c}(\lambda)|}^{-1}\:.
\end{equation*}
Next for $s \in (R_0, R)$, defining
\begin{equation*}
Pg(s):= {A(s)}^{-\frac{1}{4}} \bigintssss_{1}^\infty  {\bf c}(\lambda) \: e^{i\left\{\lambda s + t(s)\psi(\lambda)\right\}} \:g(\lambda)\: \eta_2(\lambda)\: {\left(\lambda^2 + \frac{Q^2}{4}\right)}^{-1/8} \:{|{\bf c}(\lambda)|}^{-1}\: d\lambda \:,
\end{equation*}
we note that,
\begin{equation*}
2^{m_\z/2}\: T_{\psi,5}f(s) A(s)^{1/4} = Pg(s)\:.
\end{equation*}
Thus in order to prove (\ref{annulus_pf_eq3}), it suffices to prove
\begin{equation} \label{annulus_pf_eq4}
{\left(\int_{R_0}^R {\left|Pg(s)\right|}^4\:ds\right)}^{1/4} \lesssim {\left(\int_0^\infty  {|g(\lambda)|}^2 \:d\lambda\right)}^{1/2}\:.
\end{equation}
Now for $u \in C_c^\infty(R_0,R)$ and $\lambda >0$, setting
\begin{equation*}
P^*u(\lambda)= \overline{{\bf c}(\lambda)}\:\eta_2(\lambda)\:{\left(\lambda^2 + \frac{Q^2}{4}\right)}^{-\frac{1}{8}}{|{\bf c}(\lambda)|}^{-1} \int_{R_0}^R {A(s)}^{-1/4}\: e^{-i\left\{\lambda s + t(s)\psi(\lambda)\right\}} u(s)\: ds\:,
\end{equation*}
it is easy to observe that 
\begin{equation*}
\int_{R_0}^R Pv(s) \: \overline{u(s)}\: ds = \int_{0}^\infty v(\lambda)\: \overline{P^*u(\lambda)}\: d\lambda \:,
\end{equation*}
holds for all $u \in C_c^\infty(R_0,R)$ and $v \in L^2(0,\:\infty)$ having suitable decay at infinity. Using duality  and the fact that $Supp(\eta_2) \subset \R \setminus (-1,1)$, it follows that to prove (\ref{annulus_pf_eq4}), it suffices to prove 
\begin{equation} \label{annulus_pf_eq5}
{\left(\int_1^\infty {\left|P^* h(\lambda)\right|}^2\: d\lambda\right)}^{1/2} \lesssim {\left(\int_{R_0}^R {\left|h(s)\right|}^{4/3}\:ds\right)}^{3/4}\:,
\end{equation}
for all $h \in C_c^\infty(R_0,R)$. Now expanding $\eta_2$, we get that
\begin{eqnarray*}
|P^*h(\lambda)| &=& \eta_2(\lambda)\:{\left(\lambda^2 + \frac{Q^2}{4}\right)}^{-\frac{1}{8}} \left| \int_{R_0}^R {A(s)}^{-1/4}\: e^{-i\left\{\lambda s + t(s)\psi(\lambda)\right\}} h(s)\: ds \right| \\
&=& \left(\displaystyle\sum_{k=1}^\infty \eta(2^{-k}\lambda)\:{\left(\lambda^2 + \frac{Q^2}{4}\right)}^{-\frac{1}{8}}\right) \left| \int_{R_0}^R {A(s)}^{-1/4}\: e^{-i\left\{\lambda s + t(s)\psi(\lambda)\right\}} h(s)\: ds \right| \\
& \lesssim & \left(\displaystyle\sum_{k=1}^\infty \eta(2^{-k}\lambda)\:2^{-k/4}\right) \left| \int_{R_0}^R {A(s)}^{-1/4}\: e^{-i\left\{\lambda s + t(s)\psi(\lambda)\right\}} h(s)\: ds \right| \:.
\end{eqnarray*}
For $N \in \N$, we set 
\begin{equation*}
P^*_N h(\lambda)= \left(\displaystyle\sum_{k=1}^N \eta(2^{-k}\lambda)\:2^{-k/4}\right) \left| \int_{R_0}^R {A(s)}^{-1/4}\: e^{-i\left\{\lambda s + t(s)\psi(\lambda)\right\}} h(s)\: ds \right| \:.
\end{equation*}
Then from the definition of $\eta$, it follows that
\begin{equation*}
{\left|P^*_N h(\lambda)\right|}^2 \le 3 \left(\displaystyle\sum_{k=1}^N \eta(2^{-k}\lambda)\:2^{-k/2}\right) {\left| \int_{R_0}^R {A(s)}^{-1/4}\: e^{-i\left\{\lambda s + t(s)\psi(\lambda)\right\}} h(s)\: ds \right|}^2 \:.
\end{equation*}
Then by Fubini's theorem and a change of variable,
\begin{eqnarray*}
&&\int_1^\infty {\left|P^*_N h(\lambda)\right|}^2\: d\lambda \\
& \lesssim & \int_1^\infty \left(\displaystyle\sum_{k=1}^N \eta(2^{-k}\lambda)\:2^{-k/2}\right) \left( \int_{R_0}^R {A(s)}^{-1/4}\: e^{-i\left\{\lambda s + t(s)\psi(\lambda)\right\}} h(s)\: ds \right) \\
&& \times  \left( \int_{R_0}^R {A(s')}^{-1/4}\: e^{i\left\{\lambda s' + t(s')\psi(\lambda)\right\}} \:\overline{h(s')}\: ds' \right) d\lambda \\
&=& \int_{R_0}^R \int_{R_0}^R \left(\displaystyle\sum_{k=1}^N 2^{-k/2} \int_1^\infty e^{i\left\{\lambda (s'-s) + (t(s')-t(s))\psi(\lambda)\right\}} \eta(2^{-k}\lambda)\: d\lambda  \right) \\
&& \times {A(s)}^{-1/4} \: h(s) \: {A(s')}^{-1/4} \: \overline{h(s')} \: ds \: ds' \\
&=& \int_{R_0}^R \int_{R_0}^R \left(\displaystyle\sum_{k=1}^N 2^{k/2} \int_{1/2}^2 e^{i\left\{2^k\lambda (s'-s) + (t(s')-t(s))\psi(2^k\lambda)\right\}} \eta(\lambda)\: d\lambda  \right) \\
&& \times {A(s)}^{-1/4} \: h(s) \: {A(s')}^{-1/4} \: \overline{h(s')} \: ds \: ds'\:.
\end{eqnarray*}
Now letting $N \to \infty$ and using the growth asymptotic of the density function (as $R_0 < s,s' < R$), we get that 
\begin{equation} \label{annulus_pf_eq6}
\int_1^\infty {\left|P^* h(\lambda)\right|}^2\: d\lambda \lesssim \int_{R_0}^R \int_{R_0}^R \left(\displaystyle\sum_{k=1}^\infty I_k(s,s')\right) \:|h(s)|\:|h(s')|\:s^{-\left(\frac{n-1}{4}\right)}\:{(s')}^{-\left(\frac{n-1}{4}\right)}\:ds\:ds'\:, 
\end{equation}
where for $k \ge 1$,
\begin{equation*}
I_k(s,s') = 2^{k/2} \left|\int_{1/2}^2 e^{i\left\{2^k\lambda (s'-s) + (t(s')-t(s))\psi(2^k\lambda)\right\}} \eta(\lambda)\: d\lambda \right|\:.
\end{equation*}
We now claim that
\begin{equation} \label{oscillatory_claim}
\displaystyle\sum_{k=1}^\infty I_k(s,s') \lesssim \frac{1}{{|s-s'|}^{1/2}}\:.
\end{equation}

In order to prove (\ref{oscillatory_claim}), we note that for all $k \ge 1$, one trivially has,
\begin{equation} \label{claim_pf_eq1}
I_k(s,s') \lesssim 2^{k/2}\:,
\end{equation}
where the implicit positive constant only depends on $\eta$. In the forthcoming coming computations, unless explicitly mentioned all constants are independent of $k$. For the sake of notational convenience, we write 
\begin{equation*}
d=t(s')-t(s)\:.
\end{equation*}
Without loss of generality, we will assume that $0<d<1$.

\medskip

We next recall by (\ref{phase_fn_properties}) that there exist real numbers $C_1,C_2,C_3>0$ and $\delta_2>1$ such that the first derivative of the phase function $\psi$ (for the equations we are interested in) satisfies
\begin{equation} \label{claim_pf_eq2}
 |\psi'(\lambda)| \le C_1 \lambda^{\delta_2 -1}\:,\text{  for } \lambda \ge 1  \:,
\end{equation}
and for the second derivative, one has
\begin{equation} \label{claim_pf_eq3}
C_2 \lambda^{\delta_2 -2} \le |\psi''(\lambda)| \le C_3 \lambda^{\delta_2 -2}\:,\text{  for } \lambda \ge 1  \:.
\end{equation}
We also set 
\begin{equation} \label{claim_pf_eq4}
C_4 = \displaystyle\max_{1/2 \le \lambda \le 2} \left\{\lambda^{\delta_2-1}\right\}\: \text{ and } C_5 = \frac{1}{2\max\{C_1C_4\:,2\}}\:.
\end{equation}
We first obtain the following auxiliary estimates on $I_k$,
\begin{equation} \label{claim_pf_eq5}
I_k(s,s') \lesssim 2^{k/2}\:{\left(d2^{k\delta_2}\right)}^{-1/2}\:,\:\text{ for all } k \ge 1\:,
\end{equation}
and
\begin{equation} \label{claim_pf_eq6}
I_k(s,s') \lesssim 2^{-k/2}\:{|s-s'|}^{-1}\:,\:\text{ for } d2^{k\delta_2} \le C_5\: 2^k |s-s'|\:.
\end{equation}

\medskip

We denote 
\begin{equation*}
\theta(\lambda)=2^k\lambda (s'-s) + d\psi(2^k\lambda)\:,
\end{equation*}
and then $I_k$ takes the form,
\begin{equation*}
I_k(s,s') = 2^{k/2} \left|\int_{1/2}^2 e^{i\theta(\lambda)} \eta(\lambda)\: d\lambda \right|\:.
\end{equation*}
By taking derivatives, we see that
\begin{eqnarray*}
&& \theta'(\lambda)= 2^k (s'-s) + 2^kd\:\psi'(2^k\lambda)\:, \\ 
&& \theta''(\lambda)= 2^{2k}d\:\psi''(2^k\lambda)\:.
\end{eqnarray*}
We first prove (\ref{claim_pf_eq5}). Now as $1/2 <\lambda < 2$, for $k \ge 1$, one has that $2^k \lambda >1$. Then using (\ref{claim_pf_eq3}), it follows that
\begin{equation*}
\left|\theta''(\lambda)\right| = 2^{2k}d\:\left|\psi''(2^k\lambda)\right| \ge C_2\: 2^{2k} d\:{(2^k \lambda)}^{\delta_2-2} \ge C\:d\:2^{k\delta_2}\:. 
\end{equation*}
Using the above and Lemma \ref{van_der_corput}, we obtain (\ref{claim_pf_eq5}).

\medskip

We now focus on (\ref{claim_pf_eq6}). By (\ref{claim_pf_eq2}) and (\ref{claim_pf_eq4}), we have
\begin{equation*}
|\psi'(2^k\lambda)| \le C_1 {\left(2^k\lambda\right)}^{\delta_2-1} \le C_1C_4\:2^{k(\delta_2-1)}\:,
\end{equation*}
and thus it follows that
\begin{equation*} 
2^k d \:|\psi'(2^k\lambda)| \le C_1C_4\: d\: 2^{k \delta_2}\:.
\end{equation*}
Then as we are working under the assumption that $d2^{k\delta_2} \le C_5\: 2^k |s-s'|$, we get that
\begin{equation*} 
2^k d \:|\psi'(2^k\lambda)| \le C_1C_4\: C_5\: 2^k |s-s'| \le \frac{1}{2} 2^k |s-s'|\:.
\end{equation*}
Hence, 
\begin{equation} \label{claim_pf_eq7}
|\theta'(\lambda)| \ge 2^k|s-s'| - \frac{1}{2} 2^k |s-s'| \ge \frac{1}{2} 2^k |s-s'|\:. 
\end{equation}
Also by (\ref{claim_pf_eq3}), we have
\begin{equation} \label{claim_pf_eq8}
\left|\theta''(\lambda)\right| = 2^{2k}d\:\left|\psi''(2^k\lambda)\right| \le C_3\: 2^{2k} d\:{(2^k \lambda)}^{\delta_2-2} \le C\:d\:2^{k\delta_2}\:.
\end{equation}
Now by Lemma \ref{Sjolin_lemma} (for $l=1$ in the statement),
\begin{equation*}
\int_{1/2}^2 e^{i\theta(\lambda)}\:\eta(\lambda)\:d\lambda = i \int_{1/2}^2 e^{i\theta(\lambda)}\: \left[\frac{\eta'(\lambda)}{\theta'(\lambda)}-\frac{\eta(\lambda)\theta''(\lambda)}{{\left(\theta'(\lambda)\right)}^2}\right]\: d\lambda\:.
\end{equation*}
Thus plugging (\ref{claim_pf_eq7}), (\ref{claim_pf_eq8}) and the underlying assumption in the above, we obtain
\begin{eqnarray*}
|I_k(s,s')| & \le & 2^{k/2} \left[\int_{1/2}^2 \frac{|\eta'(\lambda)|}{|\theta'(\lambda)|}\:d\lambda + \int_{1/2}^2 \frac{|\eta(\lambda)|\:|\theta''(\lambda)|}{{|\theta'(\lambda)|}^2}\: d\lambda\right] \\
& \lesssim & 2^{k/2} \left[{\left(2^k |s-s'|\right)}^{-1}+ {\left(2^k |s-s'|\right)}^{-2}\:d\:2^{k\delta_2}\right] \\
& \lesssim & 2^{-k/2}\:{|s-s'|}^{-1}\:. 
\end{eqnarray*}
This completes the proof of (\ref{claim_pf_eq6}).

\medskip

We now focus on the proof of the claim (\ref{oscillatory_claim}). Set
\begin{equation*}
C_6=C_5^{1/(\delta_2-1)}\:.
\end{equation*}
We break the proof into three cases: 

\medskip

\noindent {\bf Case 1:}\:$|s-s'| \le d^{1/\delta_2}/{C_6}$. \\
We break this case into further two subcases. First we consider the set 
\begin{equation*}
\mathscr{K}_1:=\left\{k \ge 1 \mid 2^k \le d^{-1/\delta_2}\right\}\:.
\end{equation*}
By (\ref{claim_pf_eq1}), we have
\begin{equation*}
\displaystyle\sum_{k \in \mathscr{K}_1} I_k(s,s') \lesssim \displaystyle\sum_{k \in \mathscr{K}_1} 2^{k/2} \lesssim {\left(d^{-1/\delta_2}\right)}^{1/2} \lesssim \frac{1}{{|s-s'|}^{1/2}}\:.
\end{equation*}
We now consider the complement, that is,
\begin{equation*}
\mathscr{K}_2:=\left\{k \ge 1 \mid 2^k > d^{-1/\delta_2}\right\}\:.
\end{equation*}
By (\ref{claim_pf_eq5}), it follows that
\begin{equation*}
\displaystyle\sum_{k \in \mathscr{K}_2} I_k(s,s') \lesssim  \displaystyle\sum_{k \in \mathscr{K}_2} 2^{k/2}\:{\left(d2^{k\delta_2}\right)}^{-1/2} \lesssim   d^{-1/2}\: {\left(d^{-1/\delta_2}\right)}^{\left(\frac{1}{2}-\frac{\delta_2}{2}\right)} \lesssim  \frac{1}{{|s-s'|}^{1/2}}\:.
\end{equation*}
Thus summing up both the subcases,
\begin{equation*}
\displaystyle\sum_{k=1}^\infty I_k(s,s') = \displaystyle\sum_{k \in \mathscr{K}_1} I_k(s,s') + \displaystyle\sum_{k \in \mathscr{K}_2} I_k(s,s') \lesssim \frac{1}{{|s-s'|}^{1/2}}\:.
\end{equation*}

\noindent {\bf Case 2:}\:$d^{1/\delta_2}/{C_6} \le |s-s'| <1$. \\
We start off by noting that in this case,
\begin{equation*}
\frac{1}{|s-s'|} \le C_6\:d^{-1/\delta_2} \le C_6\:{\left(\frac{|s-s'|}{d}\right)}^{1/(\delta_2-1)} \:.
\end{equation*}
We now break this case into further three subcases. We first consider the collection
\begin{equation*}
\mathscr{K}_3 := \left\{k \ge 1 \mid 2^k \le \frac{1}{|s-s'|}\right\}\:.
\end{equation*}
By (\ref{claim_pf_eq1}), we have
\begin{equation*}
\displaystyle\sum_{k \in \mathscr{K}_3} I_k(s,s') \lesssim \displaystyle\sum_{k \in \mathscr{K}_3} 2^{k/2} \lesssim \frac{1}{{|s-s'|}^{1/2}}\:.  
\end{equation*}
Next, let us consider
\begin{equation*}
\mathscr{K}_4 := \left\{k \ge 1 \mid \frac{1}{|s-s'|}< 2^k \le C_6\:{\left(\frac{|s-s'|}{d}\right)}^{1/(\delta_2-1)}\right\}\:.
\end{equation*}
As $\delta_2 >1$, we note that for $k \in \mathscr{K}_4$, we also have
\begin{equation*}
d2^{k\delta_2} \le C_5\: 2^k |s-s'|\:.
\end{equation*}
Hence, by (\ref{claim_pf_eq6}), it follows that
\begin{equation*}
\displaystyle\sum_{k \in \mathscr{K}_4} I_k(s,s') \lesssim {|s-s'|}^{-1} \displaystyle\sum_{k \in \mathscr{K}_4} 2^{-k/2} \lesssim \frac{1}{{|s-s'|}^{1/2}}\:.  
\end{equation*}
Finally, we consider 
\begin{equation*}
\mathscr{K}_5 := \left\{k \ge 1 \mid 2^k > C_6\:{\left(\frac{|s-s'|}{d}\right)}^{1/(\delta_2-1)}\right\}\:.
\end{equation*}
By (\ref{claim_pf_eq5}), we get
\begin{equation*}
\displaystyle\sum_{k \in \mathscr{K}_5} I_k(s,s') \lesssim  d^{-1/2}\displaystyle\sum_{k \in \mathscr{K}_5} 2^{k\left(\frac{1}{2}-\frac{\delta_2}{2}\right)} \lesssim d^{-1/2} {\left(\frac{|s-s'|}{d}\right)}^{\left(\frac{1}{2}-\frac{\delta_2}{2}\right)/(\delta_2-1)}=\frac{1}{{|s-s'|}^{1/2}}\:.
\end{equation*}
Thus summing up the above three subcases, we obtain
\begin{equation*}
\displaystyle\sum_{k=1}^\infty I_k(s,s') = \displaystyle\sum_{k \in \mathscr{K}_3} I_k(s,s') + \displaystyle\sum_{k \in \mathscr{K}_4} I_k(s,s') + \displaystyle\sum_{k \in \mathscr{K}_5} I_k(s,s') \lesssim \frac{1}{{|s-s'|}^{1/2}}\:.
\end{equation*}
\noindent {\bf Case 3:}\:$|s-s'| \ge 1$. \\
We first consider the collection,
\begin{equation*}
\mathscr{K}_6 := \left\{k \ge 1 \mid 2^k \le C_6\:{\left(\frac{|s-s'|}{d}\right)}^{1/(\delta_2-1)}\right\}\:.
\end{equation*}
As $\delta_2 >1$, we note that for $k \in \mathscr{K}_6$, we also have
\begin{equation*}
d2^{k\delta_2} \le C_5\: 2^k |s-s'|\:.
\end{equation*}
Hence, by (\ref{claim_pf_eq6}), it follows that
\begin{equation*}
\displaystyle\sum_{k \in \mathscr{K}_6} I_k(s,s') \lesssim {|s-s'|}^{-1} \displaystyle\sum_{k \in \mathscr{K}_6} 2^{-k/2} \lesssim \frac{1}{{|s-s'|}^{1/2}}\:.  
\end{equation*}
So the only case left is when $k \in \mathscr{K}_5$, defined above. By (\ref{claim_pf_eq5}), we get
\begin{equation*}
\displaystyle\sum_{k \in \mathscr{K}_5} I_k(s,s') \lesssim  d^{-1/2}\displaystyle\sum_{k \in \mathscr{K}_5} 2^{k\left(\frac{1}{2}-\frac{\delta_2}{2}\right)} \lesssim d^{-1/2} {\left(\frac{|s-s'|}{d}\right)}^{\left(\frac{1}{2}-\frac{\delta_2}{2}\right)/(\delta_2-1)}=\frac{1}{{|s-s'|}^{1/2}}\:.
\end{equation*}
Again summing up, we obtain
\begin{equation*}
\displaystyle\sum_{k =1}^\infty I_k(s,s')= \displaystyle\sum_{k \in \mathscr{K}_5} I_k(s,s') + \displaystyle\sum_{k \in \mathscr{K}_6} I_k(s,s') \lesssim  \frac{1}{{|s-s'|}^{1/2}}\:.
\end{equation*}
This completes the proof of the claim (\ref{oscillatory_claim}).

\medskip

Then by (\ref{annulus_pf_eq6}) and (\ref{oscillatory_claim}), it follows that
\begin{equation*}
\int_{1}^\infty {\left|P^* h(\lambda)\right|}^2\: d\lambda \lesssim \int_{R_0}^R \int_{R_0}^R \frac{|h(s)|\:|h(s')|}{{|s-s'|}^{\frac{1}{2}}} \:s^{-\left(\frac{n-1}{4}\right)}\:{(s')}^{-\left(\frac{n-1}{4}\right)}\:ds\:ds' \:.
\end{equation*}
As $h \in C^\infty_c(R_0,R)$, identifying it as an even $C^\infty_c$ function supported in $(-R,-R_0) \sqcup (R_0,R)$, we write the last integral as an one dimensional Riesz potential and apply (\ref{riesz_identity}) to get that for some $c>0$,
\begin{eqnarray*}
&& \int_{R_0}^R \int_{R_0}^R \frac{|h(s)||h(s')|}{{|s-s'|}^{1/2}}\:s^{-\left(\frac{n-1}{4}\right)}\:{(s')}^{-\left(\frac{n-1}{4}\right)}\:ds\:ds' \\
&=& \int_{0}^\infty \int_{0}^\infty \frac{|h(s)||h(s')|}{{|s-s'|}^{1/2}}\:s^{-\left(\frac{n-1}{4}\right)}\:{(s')}^{-\left(\frac{n-1}{4}\right)}\:ds\:ds' \\
&=& c\int_{0}^\infty I_{1/2} \left({(s')}^{-\left(\frac{n-1}{4}\right)} |h|\right)(s)\:s^{-\left(\frac{n-1}{4}\right)}\:|h(s)|\:ds \\
&=& c \int_{\R} {|\xi|}^{-\frac{1}{2}}\: {\left|{\left(s^{-\left(\frac{n-1}{4}\right)}|h|\right)}^{\sim}(\xi)\right|}^2\: d\xi \:.
\end{eqnarray*}
Now applying Pitt's inequality (see Lemma \ref{Pitt's_ineq}) to the function
\begin{equation*}
x\mapsto |x|^{-\frac{n-1}{4}}|h(x)|,\:\:\:\:\:\:x\in\R,
\end{equation*}
for $\beta=\frac{1}{4}$, $p=\frac{4}{3}$ and hence, $\beta_1=0$, we obtain
\begin{eqnarray*}
 \int_{\R} {|\xi|}^{-\frac{1}{2}}\: {\left|{\left(s^{-\left(\frac{n-1}{4}\right)}|h|\right)}^{\sim}(\xi)\right|}^2\: d\xi
& \lesssim & {\left(\int_{\R} {|h(x)|}^{\frac{4}{3}}\: {|x|}^{-\left(\frac{n-1}{3}\right)}\: dx\right)}^{3/2} \\
&=& c {\left(\int_{R_0}^R {|h(s)|}^{\frac{4}{3}}\: s^{-\left(\frac{n-1}{3}\right)}\: ds\right)}^{3/2} \\
&\le & c \:R^{-\left(\frac{n-1}{2}\right)}_0\: {\|h\|}^2_{L^{4/3}(R_0,R)}\:.
\end{eqnarray*}
This completes the proof of part $(a)$.

\medskip

The proof of part $(b)$ only uses the error term estimate (\ref{annulus_pf_eq2}) and the boundedness of the multiplier. Indeed, using (\ref{annulus_pf_eq2}), followed by an application of the Cauchy-Schwarz inequality, we get that for all $s\in (R_0,R)$,
\begin{eqnarray*}
|T_{\psi,7}f(s)| &\lesssim & {A(s)}^{-1/2} \int_{1}^\infty \lambda^{-1} \: \left|\widehat{f}(\lambda)\right|\:{\left|{\bf c}(\lambda)\right|}^{-1}\: d\lambda \\
& \lesssim & R^{-\left(\frac{n-1}{2}\right)}_0\: {\|f\|}_{H^0(S)}\: {\left(\int_1^\infty \frac{d\lambda}{\lambda^2}\right)}^{1/2} \\
& \lesssim &  {\|f\|}_{H^0(S)}\:.
\end{eqnarray*}
This completes the proof of Lemma \ref{large_freq_lemma_annulus_pieces} and also of Lemma \ref{large_freq_lemma_annulus}. 
\end{proof}
Consequently, we get Lemma \ref{large_freq} and this completes the proof of the sufficient conditions of Theorem \ref{theorem}.

\begin{remark} \label{remark_about_proof}
While studying the boundedness of $T_{\psi,7}$, the piece corresponding to the error term $E_2$, the pointwise decay given by (\ref{annulus_pf_eq2}) was crucial. It should be noted that the situation is worse if one proceeds with the classical Harish-Chandra series expansion in stead.
\end{remark}

\section{The transference principle}
In this section, we prove the transference principle (Theorem \ref{transference_principle}), but first we look at the following lemma which compares different Sobolev norms of radial functions with Spherical Fourier transform supported away from the origin:
\begin{lemma} \label{sobolev_comparison}
Let $f \in \mathscr{S}^2(S)_o$ with $Supp(\widehat{f}) \subset (c,\infty)$ for some $c>0$. Then for $0\le \beta_1 < \beta_2 < \infty$, we have
\begin{equation*}
\|f\|_{H^{\beta_1}(S)} \le c^{-(\beta_2-\beta_1)}\: \|f\|_{H^{\beta_2}(S)}\:.
\end{equation*}
\end{lemma}
\begin{proof}
The Lemma follows from a straightforward computation. Indeed,
\begin{eqnarray*}
\|f\|_{H^{\beta_1}(S)} &=& {\left(\int_c^\infty {\left(\lambda^2 + \frac{Q^2}{4}\right)}^{\beta_1}\: {|\widehat{f}(\lambda)|}^2\: {|{\bf c}(\lambda)|}^{-2}\: d\lambda\right)}^{1/2} \\
&=& {\left(\int_c^\infty {\left(\lambda^2 + \frac{Q^2}{4}\right)}^{-(\beta_2- \beta_1)}\: {\left(\lambda^2 + \frac{Q^2}{4}\right)}^{\beta_2}\:{|\widehat{f}(\lambda)|}^2\: {|{\bf c}(\lambda)|}^{-2}\: d\lambda\right)}^{1/2} \\
&\le & {\left(\int_c^\infty \lambda^{-2(\beta_2-\beta_1)}\: {\left(\lambda^2 + \frac{Q^2}{4}\right)}^{\beta_2}\:{|\widehat{f}(\lambda)|}^2\: {|{\bf c}(\lambda)|}^{-2}\: d\lambda\right)}^{1/2} \\
& \le & c^{-(\beta_2-\beta_1)}\:\|f\|_{H^{\beta_2}(S)}\:.  
\end{eqnarray*}
\end{proof}

\begin{proof}[Proof of Theorem \ref{transference_principle}]
Let us assume that for all $R>0$, some $\beta_0>0$ and some $p \in [1,\infty]$, the maximal estimate 
\begin{equation} \label{trans_pf_eq1}
{\|S^*_{\psi_1} f\|}_{L^p(B_R)} \lesssim {\|f\|}_{H^\beta(S)}\:,
\end{equation}
holds for all $\beta >\beta_0$ and all $f \in \mathscr{S}^2(S)_o$. We have to show that for all $R>0$, the maximal estimate 
\begin{equation} \label{trans_pf_eq2}
{\|S^*_{\psi_2} f\|}_{L^p(B_R)} \lesssim {\|f\|}_{H^\beta(S)}\:,
\end{equation}
holds for all $\beta >\beta_0$ and all $f \in \mathscr{S}^2(S)_o$.

\medskip

By the hypothesis, there exist positive constants $\Lambda$ and $C_7$ such that $\psi_1,\psi_2$, the phase functions of the corresponding multipliers  satisfy
\begin{equation} \label{trans_pf_eq3}
|\psi_1(\lambda)-\psi_2(\lambda)| \le C_7
\:,\:\text{ for } \lambda \in (\Lambda,\infty)\:,
\end{equation}
and $\psi_1,\psi_2 \in C^\infty(\Lambda,\infty)$. Now let us choose and fix $R>0$ and $f \in \mathscr{S}^2(S)_o$. We next consider an auxiliary non-negative, even function $\eta \in C^\infty_c(\R)$ such that $Supp(\eta) \subset \left\{\xi : 1/2 < |\xi| <2\right\}$ and 
\begin{equation*}
\displaystyle\sum_{k=-\infty}^\infty \eta \left(2^{-k} \xi\right)=1\:,\: \text{ for } \xi \ne 0\:.
\end{equation*}
Let $k_0$ be the largest $k \in \Z$ such that $k \le \log_2 \lceil \Lambda \rceil +1$. Then setting 
\begin{equation*}
\eta_{k_0}(\xi):= \displaystyle\sum_{k=-\infty}^{k_0} \eta \left(2^{-k} \xi\right)\:,
\end{equation*}
and for each $k > k_0$, setting
\begin{equation*}
\eta_k(\xi):= \eta\left(2^{-k}\xi\right)\:,
\end{equation*}
we get,
\begin{equation*}
\widehat{f} \equiv \widehat{f} \cdot \eta_{k_0} + \displaystyle\sum_{k>k_0} \widehat{f} \cdot \eta_k\:.
\end{equation*}
By the Schwartz isomorphism theorem, for each $k \ge k_0$, there exists $f_k \in \mathscr{S}^2(S)_o$ such that $\widehat{f_k}=\widehat{f} \cdot \eta_k$ and hence,
\begin{equation} \label{trans_pf_eq4}
\widehat{f} \equiv \widehat{f_{k_0}}+ \displaystyle\sum_{k>k_0} \widehat{f_k}\:,
\end{equation}
with $Supp\left(\widehat{f_{k_0}}\right) \subset \left[0,2^{k_0+1}\right)$ and for $k>k_0$, $Supp\left(\widehat{f_k}\right) \subset \left(2^{k-1},2^{k+1}\right)$. Then (\ref{trans_pf_eq4}) yields,
\begin{equation} \label{trans_pf_eq5}
{\|S^*_{\psi_2} f\|}_{L^p(B_R)} \le {\|S^*_{\psi_2} f_{k_0}\|}_{L^p(B_R)} + \displaystyle\sum_{k>k_0} {\|S^*_{\psi_2} f_k\|}_{L^p(B_R)}\:.
\end{equation}
Now proceeding as in the proof of the small frequency case (Lemma \ref{small_freq}), we have
\begin{equation} \label{trans_pf_eq6}
{\|S^*_{\psi_2} f_{k_0}\|}_{L^p(B_R)} \lesssim {\|f_{k_0}\|}_{H^\beta(S)} \le {\|f\|}_{H^\beta(S)} \:, \text{ for all } \beta>\beta_0\:.
\end{equation}
Next for each $k>k_0$, by using the formula (\ref{dispersive_soln}) and Taylor's formula, we have for $s \in [0,R]$ and $t \in (0,1)$,
\begin{eqnarray} \label{trans_pf_eq7}
&&\left| S_{\psi_1,t}f_k(s) - S_{\psi_2,t}f_k(s) \right| \nonumber\\
&=& \left| \int_{0}^\infty \varphi_\lambda(s)\:\left(e^{it\psi_1(\lambda)}-e^{it\psi_2(\lambda)}\right)\:\widehat{f_k}(\lambda)\: {|{\bf c}(\lambda)|}^{-2}\: d\lambda\:\right| \nonumber\\
&=& \left| \int_{0}^\infty \varphi_\lambda(s)\:e^{it\psi_1(\lambda)}\:\left(1-e^{it(\psi_2(\lambda)-\psi_1(\lambda))}\right)\:\widehat{f_k}(\lambda)\: {|{\bf c}(\lambda)|}^{-2}\: d\lambda\:\right| \nonumber\\
&=&\left| \int_{0}^\infty \varphi_\lambda(s)\:e^{it\psi_1(\lambda)}\:\left[1-\displaystyle\sum_{j=0}^\infty \frac{{\left\{it(\psi_2(\lambda)-\psi_1(\lambda))\right\}}^j}{j!}\right]\:\widehat{f_k}(\lambda)\: {|{\bf c}(\lambda)|}^{-2}\: d\lambda\:\right| \nonumber\\
& \le & \displaystyle\sum_{j=1}^\infty \frac{t^j}{j!} \left| \int_{0}^\infty \varphi_\lambda(s)\:e^{it\psi_1(\lambda)}\:\left[ \psi_2(\lambda)-\psi_1(\lambda)\right]^j\:\widehat{f_k}(\lambda)\: {|{\bf c}(\lambda)|}^{-2}\: d\lambda\:\right|\:.
\end{eqnarray}
Now again by the Schwartz isomorphism theorem, for each $j \in \N$, there exists $g_{j,k} \in \mathscr{S}^2(S)_o$ such that 
\begin{equation*}
\widehat{g_{j,k}} = \left[ \psi_2-\psi_1\right]^j\:\widehat{f_k}\:.
\end{equation*} 
Then again using the formula (\ref{dispersive_soln}), we obtain from (\ref{trans_pf_eq7}),
\begin{equation} \label{trans_pf_eq8}
\left\|\displaystyle\sup_{0<t<1}\left| S_{\psi_1,t}f_k - S_{\psi_2,t}f_k \right|\right\|_{L^p(B_R)} \le \displaystyle\sum_{j=1}^\infty \frac{1}{j!} {\left\|S^*_{\psi_1} g_{j,k}\right\|}_{L^p(B_R)}\:. 
\end{equation}
For any $\varepsilon>0$, we set $\beta_1=\beta_0 +\frac{\varepsilon}{2}$ and $\beta_2=\beta_0 +\varepsilon$. Then noting that $Supp(\widehat{g_{j,k}}) \subset \left(2^{k-1},2^{k+1}\right)$ and applying (\ref{trans_pf_eq1}) and Lemma \ref{sobolev_comparison}, we get that
\begin{equation*}
{\left\|S^*_{\psi_1} g_{j,k}\right\|}_{L^p(B_R)} \lesssim {\|g_{j,k}\|}_{H^{\beta_1}(S)} \lesssim 2^{-\frac{k\varepsilon}{2}}\:  {\|g_{j,k}\|}_{H^{\beta_2}(S)} \:,
\end{equation*} 
where the resultant implicit constant is independent of $k$. Plugging the above in (\ref{trans_pf_eq8}),
\begin{equation} \label{trans_pf_eq9}
\left\|\displaystyle\sup_{0<t<1}\left| S_{\psi_1,t}f_k - S_{\psi_2,t}f_k \right|\right\|_{L^p(B_R)} \lesssim 2^{-\frac{k\varepsilon}{2}} \displaystyle\sum_{j=1}^\infty \frac{1}{j!} {\|g_{j,k}\|}_{H^{\beta_2}(S)}\:. 
\end{equation}
Next using the definition of $g_{j,k}$ and (\ref{trans_pf_eq3}), it follows that
\begin{equation*}
{\|g_{j,k}\|}_{H^{\beta_2}(S)} \le C_7^j\: {\|f_k\|}_{H^{\beta_2}(S)}\:.
\end{equation*}
Plugging the above in (\ref{trans_pf_eq9}), we have
\begin{eqnarray} \label{trans_pf_eq10}
\left\|\displaystyle\sup_{0<t<1}\left| S_{\psi_1,t}f_k - S_{\psi_2,t}f_k \right|\right\|_{L^p(B_R)} &\lesssim & 2^{-\frac{k\varepsilon}{2}} \displaystyle\sum_{j=1}^\infty \frac{C_7^j}{j!} {\|f_k\|}_{H^{\beta_2}(S)} \nonumber\\
&\lesssim & 2^{-\frac{k\varepsilon}{2}} {\|f_k\|}_{H^{\beta_2}(S)}\:.
\end{eqnarray}
Hence for each $k>k_0$ using (\ref{trans_pf_eq1}), (\ref{trans_pf_eq10}) and Lemma \ref{sobolev_comparison}, it follows that
\begin{eqnarray*}
{\|S^*_{\psi_2} f_k\|}_{L^p(B_R)} &\le & \left\|\displaystyle\sup_{0<t<1}\left| S_{\psi_1,t}f_k - S_{\psi_2,t}f_k \right|\right\|_{L^p(B_R)} + {\|S^*_{\psi_1} f_k\|}_{L^p(B_R)} \\
& \lesssim & 2^{-\frac{k\varepsilon}{2}} {\|f_k\|}_{H^{\beta_2}(S)} \:+\: {\|f_k\|}_{H^{\beta_1}(S)} \\
& \lesssim & 2^{-\frac{k\varepsilon}{2}} {\|f_k\|}_{H^{\beta_2}(S)} \\
& \le & 2^{-\frac{k\varepsilon}{2}} {\|f\|}_{H^{\beta_2}(S)}\:.
\end{eqnarray*}
Then plugging the above and (\ref{trans_pf_eq6}) in (\ref{trans_pf_eq5}), we get that
\begin{eqnarray*}
{\|S^*_{\psi_2} f\|}_{L^p(B_R)} &\le & {\|S^*_{\psi_2} f_{k_0}\|}_{L^p(B_R)} + \displaystyle\sum_{k>k_0} {\|S^*_{\psi_2} f_k\|}_{L^p(B_R)}\\
& \lesssim & \left(1+ \displaystyle\sum_{k>k_0} 2^{-\frac{k\varepsilon}{2}} \right){\|f\|}_{H^{\beta_2}(S)} \\
& \lesssim & {\|f\|}_{H^{\beta_2}(S)}\:.
\end{eqnarray*}
Now as $R>0,\:\varepsilon>0$ and $f \in \mathscr{S}^2(S)_o$ were arbitrarily chosen, (\ref{trans_pf_eq2}) follows. This completes the proof of Theorem \ref{transference_principle}.
\end{proof}

\section{Proof of the necessary conditions of Theorem \ref{theorem}}
We present the proof of necessity of the conditions by dividing it into three different cases depending on the value/range of values of the Sobolev exponent $\beta$.

\medskip
{\bf Case 1}: {\em If $\beta < 1/4$, then the inequality (\ref{estimates_on_balls}) holds for no $p \in [1,\infty]$.}
\medskip
\noindent

For $f \in \mathscr{S}^2(S)_o$, we recall the linearized maximal function,
\begin{equation*}
T_\psi f(s) = \int_{0}^\infty \varphi_\lambda(s)\:e^{it(s)\psi(\lambda)}\:\widehat{f}(\lambda)\: {|{\bf c}(\lambda)|}^{-2}\: d\lambda\:.
\end{equation*}
It suffices to prove that there is no inequality of the form,
\begin{equation} \label{first_example_eq1}
\int_0^1 \left|T_\psi f(s)\right|\:A(s)\:ds \lesssim {\|f\|}_{H^\beta(S)}\:,
\end{equation}
for any $\beta <1/4$. 

\medskip

We first work out the case for the Fractional Schr\"odinger equation (for $a>1$) corresponding to $\Delta$, that is, 
\begin{equation*}
\psi(\lambda) = {\left(\lambda^2 + \frac{Q^2}{4}\right)}^{a/2}\:.
\end{equation*}
Let $\eta \in C^\infty_c(\R)_e$ be (non-zero) non-negative with $Supp(\eta) \subset (-1,1)$ and for large $N \in \N$, choose $f_N \in \mathscr{S}^2(S)_o$ such that 
\begin{equation*}
\widehat{f_N}(\lambda)=N^{-1/2}\: \eta \left(-N^{-1/2} \lambda + N^{1/2}\right)|{\bf c}(\lambda)|\:.
\end{equation*}
It follows that $Supp(\widehat{f_N}) \subset [N-N^{1/2}\:,\:N+N^{1/2}]$ and 
\begin{eqnarray*}
{\|f_N\|}_{H^\beta(S)} &=& {\left(\int_0^\infty {\left(\lambda^2 + \frac{Q^2}{4}\right)}^\beta \: \widehat{f_N}(\lambda)^2 \: {|{\bf c}(\lambda)|}^{-2} d\lambda\right)}^{1/2} \\
&=& {\left(\frac{1}{N}\int_{N-N^{1/2}}^{N+N^{1/2}} {\left(\lambda^2 + \frac{Q^2}{4}\right)}^\beta \: \eta \left(-N^{-1/2} \lambda + N^{1/2}\right)^2 \: d\lambda\right)}^{1/2} \\
& \lesssim & {\left(\frac{1}{N}\int_{N-N^{1/2}}^{N+N^{1/2}} N^{2\beta} \: d\lambda\right)}^{1/2} \\
& \lesssim & N^{\beta - \frac{1}{4}}\:.
\end{eqnarray*}
Thus, we have
\begin{equation} \label{first_example_eq2}
{\|f_N\|}_{H^\beta(S)} \to 0\:,\:\text{ as } N \to \infty\:,\:\text{ if } \beta <1/4\:.
\end{equation}
We now aim to show that for $\varepsilon>0$ sufficiently small, there exists $N_0 \in \N$ such that for all $N > N_0$, there exists a constant $c>0$, independent of $N$ so that
\begin{equation} \label{first_example_eq3}
|T_\psi f_N(s)| >c\:,\:\text{ for } s \in [\varepsilon,2\varepsilon]\:.
\end{equation}
Then in view of (\ref{first_example_eq2}) and (\ref{first_example_eq3}), we note that (\ref{first_example_eq1}) must fail. Therefore it suffices to prove (\ref{first_example_eq3}).

\medskip

For $\varepsilon>0$ sufficiently small and $s \in [\varepsilon,2\varepsilon]$, by putting $M=0$ in Lemma \ref{bessel_series_expansion}, we expand $\varphi_\lambda$ in the Bessel series: 
\begin{equation} \label{first_example_eq4}
\varphi_\lambda(s)= c_0 {\left(\frac{s^{n-1}}{A(s)}\right)}^{1/2} \J_{\frac{n-2}{2}}(\lambda s) + E_1(\lambda,s)\:.
\end{equation}
We recall the definition,
\begin{equation} \label{first_example_eq5}
\J_{\frac{n-2}{2}}(\lambda s)= 2^{\frac{n-2}{2}} \: \pi^{1/2} \: \Gamma\left(\frac{n-1}{2}\right) \frac{J_{\frac{n-2}{2}}(\lambda s)}{(\lambda s)^{\left(\frac{n-2}{2}\right)}}\:.
\end{equation}
Now choosing $N_1 \in \N$ so that
\begin{equation*}
N_1 - N^{1/2}_1 > \frac{B_n}{\varepsilon}\:,
\end{equation*}
where $B_n=\max\left\{A_{\frac{n-2}{2}},1\right\}$ ($A_{\frac{n-2}{2}}$ is as in Lemma \ref{bessel_function_expansion}), note that for $N \ge N_1$,  we have for $\lambda \in Supp(\widehat{f_N})$,
\begin{equation*}
\frac{B_n}{s} \le \frac{B_n}{\varepsilon} < N - N^{1/2} \le \lambda \:.
\end{equation*} 
Thus for $s \in [\varepsilon,2 \varepsilon]$ and $\lambda \in Supp(\widehat{f_N})$, we get that
\begin{equation*}
\left|E_1(\lambda,s)\right| \lesssim s^2\: {(\lambda s)}^{-\frac{n+1}{2}}\:.
\end{equation*}
Moreover, by Lemma \ref{bessel_function_expansion},
\begin{eqnarray} \label{first_example_eq6}
J_{\frac{n-2}{2}}(\lambda s) &=& \sqrt{\frac{2}{\pi}} \frac{1}{{(\lambda s)}^{1/2}} \cos \left(\lambda s - \frac{\pi}{2}\left(\frac{n-2}{2}\right) - \frac{\pi}{4}\right) + \tilde{E}_{\frac{n-2}{2}}(\lambda s) \nonumber \\
&=& \sqrt{\frac{1}{2\pi}} \frac{1}{{(\lambda s)}^{1/2}} \left\{e^{i\left(\lambda s - \frac{\pi}{4}(n-1)\right)} + e^{-i\left(\lambda s - \frac{\pi}{4}(n-1)\right)}\right\} + \tilde{E}_{\frac{n-2}{2}}(\lambda s)\:,
\end{eqnarray}
where
\begin{equation*}
\left|\tilde{E}_{\frac{n-2}{2}}(\lambda s)\right| \lesssim \frac{1}{{(\lambda s)}^{3/2}}\:.
\end{equation*}
Then pluggining (\ref{first_example_eq6}) and (\ref{first_example_eq5}) in (\ref{first_example_eq4}), we obtain 
\begin{equation} \label{first_example_eq7}
\varphi_\lambda(s)= c {\left(\frac{s^{n-1}}{A(s)}\right)}^{1/2} \frac{1}{{(\lambda s)}^{\frac{n-1}{2}}} \left\{e^{i\left(\lambda s - \frac{\pi}{4}(n-1)\right)} + e^{-i\left(\lambda s - \frac{\pi}{4}(n-1)\right)}\right\} + \mathscr{E}_1(\lambda,s)\:,
\end{equation}
where
\begin{equation*}
\mathscr{E}_1(\lambda,s) = E_1(\lambda,s) + c {\left(\frac{s^{n-1}}{A(s)}\right)}^{1/2} \:\frac{\tilde{E}_{\frac{n-2}{2}}(\lambda s)}{(\lambda s)^{\left(\frac{n-2}{2}\right)}}\:,
\end{equation*}
and thus
\begin{eqnarray} \label{first_example_eq8}
\left|\mathscr{E}_1(\lambda,s)\right| & \lesssim &  \left\{s^2 {(\lambda s)}^{-\frac{n+1}{2}} + {(\lambda s)}^{-\frac{n+1}{2}} \right\} \nonumber \\
& \lesssim & {(\lambda s)}^{-\frac{n+1}{2}}\:.
\end{eqnarray}

Thus for $s \in [\varepsilon, 2\varepsilon]$ and $\lambda \in Supp(\widehat{f_N})$, we invoke the decomposition of $\varphi_\lambda$ given by (\ref{first_example_eq7}) to write
\begin{equation} \label{first_example_eq9}
T_{\psi}f_N(s)= U_1f_N(s) + U_2f_N(s) + U_3f_N(s)\:,
\end{equation}
where
\begin{eqnarray*}
U_1f_N(s)&=& c \bigintssss_{N-N^{1/2}}^{N+N^{1/2}} {\left(\frac{s^{n-1}}{A(s)}\right)}^{1/2} \frac{1}{{(\lambda s)}^{\frac{n-1}{2}}} \: e^{i\left(\lambda s - \frac{\pi}{4}(n-1)\right)}  \:e^{it(s){\left(\lambda^2 + \frac{Q^2}{4}\right)}^{a/2}}\:\widehat{f_N}(\lambda)\: {|{\bf c}(\lambda)|}^{-2}\: d\lambda\:, \\
U_2f_N(s)&=& c \bigintssss_{N-N^{1/2}}^{N+N^{1/2}} {\left(\frac{s^{n-1}}{A(s)}\right)}^{1/2} \frac{1}{{(\lambda s)}^{\frac{n-1}{2}}} \: e^{-i\left(\lambda s - \frac{\pi}{4}(n-1)\right)}  \:e^{it(s){\left(\lambda^2 + \frac{Q^2}{4}\right)}^{a/2}}\:\widehat{f_N}(\lambda)\: {|{\bf c}(\lambda)|}^{-2}\: d\lambda \:,\\
U_3f_N(s)&=& \bigintssss_{N-N^{1/2}}^{N+N^{1/2}} \mathscr{E}_1(\lambda,s) \:e^{it(s){\left(\lambda^2 + \frac{Q^2}{4}\right)}^{a/2}}\:\widehat{f_N}(\lambda)\: {|{\bf c}(\lambda)|}^{-2}\: d\lambda \:.
\end{eqnarray*}

By plugging in the definition of $\widehat{f_N}$ in terms of $\eta$ and then performing the change of variable $\xi=-N^{-1/2} \lambda + N^{1/2}$, we get
\begin{eqnarray*}
&&U_2f_N(s)\\
&=& \frac{ce^{i\frac{\pi}{4}(n-1)}}{{A(s)}^{1/2}} \bigintssss_{N-N^{1/2}}^{N+N^{1/2}} e^{-i\left\{\lambda s -t(s){\left(\lambda^2 + \frac{Q^2}{4}\right)}^{a/2} \right\}}  \:\frac{N^{-1/2}\: \eta \left(-N^{-1/2} \lambda + N^{1/2}\right)}{\lambda^{\frac{n-1}{2}}}\: {|{\bf c}(\lambda)|}^{-1}\: d\lambda \\
&=& \frac{ce^{i\frac{\pi}{4}(n-1)}}{{A(s)}^{1/2}} \bigintssss_{-1}^1 e^{i \theta(\xi)}\:\zeta(\xi)\:d\xi\:,
\end{eqnarray*}
where 
\begin{eqnarray*}
&& \theta(\xi)= \left(N^{1/2} \xi - N\right)s + t(s) \left[(N-N^{1/2} \xi)^2 + \frac{Q^2}{4}\right]^{a/2} \:, \\
&& \zeta(\xi)= \eta(\xi)\: \frac{{\left|{\bf c}\left(N-N^{1/2}\xi\right)\right|}^{-1}}{\left(N-N^{1/2}\xi\right)^{\frac{n-1}{2}}}\:.
\end{eqnarray*}
Now we note that
\begin{equation*}
\theta(\xi)= N^{1/2} \xi s- Ns + t(s)N^a \left[1-2N^{-1/2}\xi + N^{-1}\xi^2 + \frac{N^{-2}Q^2}{4}\right]^{a/2} \:.
\end{equation*}
As $|\xi| \le 1$, there exists $N_2 \in \N$ such that for all $N> N_2$,
\begin{equation*}
\left|-2N^{-1/2}\xi + N^{-1}\xi^2 + \frac{N^{-2}Q^2}{4}\right| \le 2N^{-1/2} + N^{-1} + \frac{N^{-2}Q^2}{4} < 1\:,
\end{equation*}
and thus applying the generalized Binomial expansion along with $|\xi| \le 1$, we get that
\begin{eqnarray*}
\theta(\xi)&=& N^{1/2} \xi s- Ns + t(s)N^a \left[1-aN^{-1/2}\xi + \frac{a(a-1)}{2}N^{-1}\xi^2 + \mathcal{O}\left(N^{-3/2}\right)\right] \\
&=& t(s)N^a - Ns + N^{1/2} \xi s - t(s)aN^{a-\frac{1}{2}} \xi + \frac{a(a-1)}{2}t(s) N^{a-1}\xi^2 + \mathcal{O}\left(t(s)N^{a-\frac{3}{2}}\right)\:.
\end{eqnarray*}

\noindent Now choosing, 
\begin{equation*}
t(s)=\frac{s}{aN^{a-1}}\:,
\end{equation*}
we get rid of the linear terms in $\xi$ to obtain,
\begin{equation*}
\theta(\xi)=\frac{a-1}{2}s\xi^2 + C(a,N,s)\:.
\end{equation*}
Thus for $s \in [\varepsilon\:,\:2\varepsilon]$, using the growth asymptotic of the density function (\ref{density_function}), it follows that
\begin{equation*}
\left|U_2f_N(s)\right| \ge c\: \varepsilon^{-\frac{n-1}{2}} \left|\int_{-1}^1 e^{i\frac{(a-1)}{2}s\xi^2}\:\zeta(\xi)\:d\xi\right|\:.
\end{equation*}
Now recalling the definition of $\zeta$, we note that $\eta$ is non-negative and as $\xi \in [-1,1]$, 
\begin{equation*}
N-N^{1/2} \le N-N^{1/2}\xi \le N+N^{1/2}\:,
\end{equation*} 
and thus
\begin{equation*}
\frac{{\left|{\bf c}\left(N-N^{1/2}\xi\right)\right|}^{-1}}{\left(N-N^{1/2}\xi\right)^{\frac{n-1}{2}}} \asymp 1\:.
\end{equation*}
Then using the above information, expanding the last integral in terms of sine and cosine and using their properties along with the fact that $s \in [\varepsilon, 2\varepsilon]$ for $\varepsilon>0$ small, it easily follows that for $N > \max\{N_1,N_2\}$, there exists a positive constant $C_8$, independent of $N$ such that 
\begin{equation} \label{first_example_eq10}
\left|U_2f_N(s)\right| > C_8\:,\:\:\:\: \text{for  } s \in [\varepsilon,2\varepsilon]\:.
\end{equation} 

We next show that for large $N$, the growths of both $U_1f_N$ and $U_3f_N$ are subordinated by $U_2f_N$. We first focus on $U_3f_N$. For $s \in [\varepsilon,2\varepsilon]$, using the error term estimates (\ref{first_example_eq8}), we get that there exists $N_3 \in \N$ such that for all $N \ge N_3$,
\begin{eqnarray} \label{first_example_eq11}
\left|U_3f_N(s)\right| &\le & \frac{c}{N^{1/2}}\bigintssss_{N-N^{1/2}}^{N+N^{1/2}}(\lambda s)^{-\frac{n+1}{2}} \:\eta\left(-N^{-1/2}\lambda + N^{1/2}\right)\: {|{\bf c}(\lambda)|}^{-1}\: d\lambda \nonumber\\
&\le & \frac{c\:\varepsilon^{-\frac{n+1}{2}}}{N^{1/2}} \bigintssss_{N-N^{1/2}}^{N+N^{1/2}} \frac{d\lambda}{\lambda} \nonumber\\
& \le & \frac{c\:\varepsilon^{-\frac{n+1}{2}}}{N} \nonumber\\
& < & \frac{C_8}{4}\:.
\end{eqnarray}
The inequality in the third line follows from the elementary estimate
\begin{equation*}
\log\left(\frac{x+1}{x-1}\right) \lesssim \frac{1}{x}\:,
\end{equation*}
which is valid for all large $x$.

\medskip

We now shift our focus to $U_1f_N$. Again proceeding as in the case of $U_2f_N$, we get
\begin{equation*}
U_1f_N(s)= \frac{ce^{-i\frac{\pi}{4}(n-1)}}{{A(s)}^{1/2}} \bigintssss_{-1}^1 e^{i \theta(\xi)}\:\zeta(\xi)\:d\xi\:,
\end{equation*}
where 
\begin{eqnarray*}
&& \theta(\xi)= -N^{1/2} s\xi + Ns + t(s) \left[(N-N^{1/2} \xi)^2 + \frac{Q^2}{4}\right]^{a/2} \:, \\
&& \zeta(\xi)= \eta(\xi)\: \frac{{\left|{\bf c}\left(N-N^{1/2}\xi\right)\right|}^{-1}}{\left(N-N^{1/2}\xi\right)^{\frac{n-1}{2}}}\:.
\end{eqnarray*}
Now by Lemma \ref{Sjolin_lemma} (for $l=1$ in the statement),
\begin{equation} \label{first_example_eq12}
\int_{-1}^1 e^{i\theta(\xi)}\:\zeta(\xi)\:d\xi = i \int_{-1}^1 e^{i\theta(\xi)}\: \left[\frac{\zeta'(\xi)}{\theta'(\xi)}-\frac{\zeta(\xi)\theta''(\xi)}{{\left(\theta'(\xi)\right)}^2}\right]\: d\xi\:.
\end{equation}
We note that
\begin{equation*}
\theta'(\xi)= -N^{1/2} \left[s+t(s)\:a \left(N-N^{1/2}\xi\right)\left\{\left(N-N^{1/2}\xi\right)^2 + \frac{Q^2}{4}\right\}^{\left(\frac{a}{2}-1\right)} \right]\:.
\end{equation*}
Now as $\xi \in [-1,1]$, we have that 
\begin{equation*}
N-N^{1/2}\xi \ge 0\:,
\end{equation*}
and thus for $s \in [\varepsilon,2\varepsilon]$,
\begin{equation*}
s+t(s)\:a \left(N-N^{1/2}\xi\right)\left\{\left(N-N^{1/2}\xi\right)^2 + \frac{Q^2}{4}\right\}^{\left(\frac{a}{2}-1\right)} \ge \varepsilon \:.
\end{equation*}
This implies that
\begin{equation} \label{first_example_eq13}
|\theta'(\xi)| \ge \varepsilon N^{1/2}\:.
\end{equation}
We next note that
\begin{equation*}
\theta''(\xi)= t(s)aN \left\{\left(N-N^{1/2}\xi\right)^2 + \frac{Q^2}{4}\right\}^{\left(\frac{a}{2}-2\right)} \left\{(a-1)\left(N-N^{1/2}\xi\right)^2 + \frac{Q^2}{4}\right\}\:.
\end{equation*}
Recalling that $\xi \in [-1,1]$, $t(s)=s/(aN^{a-1})$ and $s \in [\varepsilon,2\varepsilon]$, we get that there exists $N_4 \in \N$ such that for all $N > N_4$,
\begin{equation} \label{first_example_eq14}
|\theta''(\xi)| \le C N^{1-a}NN^{a-4}N^2=C\:.
\end{equation}
Now recalling the definition of $\zeta$, we note that as $\xi \in [-1,1]$, 
\begin{equation*}
N-N^{1/2} \le N-N^{1/2}\xi \le N+N^{1/2}\:,
\end{equation*} 
and thus
\begin{equation*}
\frac{{\left|{\bf c}\left(N-N^{1/2}\xi\right)\right|}^{-1}}{\left(N-N^{1/2}\xi\right)^{\frac{n-1}{2}}} \asymp 1\:.
\end{equation*}
It follows that
\begin{equation} \label{first_example_eq15}
|\zeta(\xi)| \le C\:.
\end{equation}
We next note that
\begin{eqnarray*}
\zeta'(\xi) &=& \eta'(\xi)\:\frac{{\left|{\bf c}\left(N-N^{1/2}\xi\right)\right|}^{-1}}{\left(N-N^{1/2}\xi\right)^{\frac{n-1}{2}}}\: +\: \eta(\xi)\: \frac{\frac{d}{d\xi}\left({\left|{\bf c}\left(N-N^{1/2}\xi\right)\right|}^{-1}\right)}{\left(N-N^{1/2}\xi\right)^{\frac{n-1}{2}}} \\
&&+ \frac{(n-1)N^{1/2}\eta(\xi)\:{\left|{\bf c}\left(N-N^{1/2}\xi\right)\right|}^{-1}}{2\left(N-N^{1/2}\xi\right)^{\frac{n+1}{2}}}
\end{eqnarray*}
Then using the pointwise and derivative estimates of ${\left|\bf c(\cdot)\right|}^{-2}$ given in (\ref{plancherel_measure}) and (\ref{c-fn_derivative_estimates}), we get that there exists $N_5 \in \N$ such that for all $N>N_5$,
\begin{equation} \label{first_example_eq16}
|\zeta'(\xi)| \le C \left(1+ N^{-1/2} + N^{-1/2}\right) \le C\:.
\end{equation}
Thus for $s \in [\varepsilon,2\varepsilon]$, using (\ref{first_example_eq12})-(\ref{first_example_eq16}), there exists $N_6 \in \N$ such that for all $N>N_6$,
\begin{eqnarray} \label{first_example_eq17}
\left|U_1f_N(s)\right| &\le & C\:\varepsilon^{-\frac{n-1}{2}} \left|\bigintssss_{-1}^1 e^{i \theta(\xi)}\:\zeta(\xi)\:d\xi\right| \nonumber \\
&\le &  C\:\varepsilon^{-\frac{n-1}{2}}\:N^{-1/2} \nonumber\\
& < & \frac{C_8}{4} \:.
\end{eqnarray}
Hence for $s \in [\varepsilon,2 \varepsilon]$, setting $N_0=\displaystyle\max_{1 \le j \le 6}N_j$, for all $N>N_0$, in view of the decomposition (\ref{first_example_eq9}) we get from (\ref{first_example_eq10}), (\ref{first_example_eq11}) and (\ref{first_example_eq17}) that
\begin{equation*}
\left|T_\psi f_N(s)\right| \ge  \left|U_2f_N(s)\right| - \left|U_1f_N(s)\right| - \left|U_3f_N(s)\right| > \frac{C_8}{2}\:. 
\end{equation*}
This yields (\ref{first_example_eq3}) and hence the proof of case 1 for the Fractional Schr\"odinger equation (for $a>1$) corresponding to $\Delta$, is completed.

\medskip

The above arguments can be carried out verbatim for the Fractional Schr\"odinger equation (for $a>1$) corresponding to the shifted operator $\tilde{\Delta}$ (perhaps even a bit simpler due to no spectral gap).

\medskip

The case of the Boussinesq equation and the Beam equation (both corresponding to $\Delta$ as well as $\tilde{\Delta}$) follow from the result for the Schr\"odinger equation which we just proved and the Transference principle (Theorem \ref{transference_principle}). Indeed, in their case, if the estimate (\ref{estimates_on_balls}) is true for some $\beta_0<1/4$, then it would also be true for any $\beta>\beta_0$. Then by point $(ii)$ of Remark \ref{examples_remark} and the transference principle (Theorem \ref{transference_principle}), (\ref{estimates_on_balls}) is also true for the classical Schr\"odinger equation ($a=2$) for any $\beta>\beta_0$ and hence in particular for the choice 
\begin{equation*}
\beta = \frac{1}{2}\left(\beta_0 + \frac{1}{4}\right) < \frac{1}{4}\:.
\end{equation*} 
But that is a contradiction. This completes the proof of case 1.

\medskip
{\bf Case 2}: {\em If $1/4 \le \beta < n/2$, then the condition $p \le 2n/(n-2\beta)$ is necessary for validity of the inequality (\ref{estimates_on_balls}).}

\medskip
\noindent
Let $\eta \in C^\infty_c(\R)_{e}$ be (non-zero) non-negative with $Supp(\eta) \subset (1,2)$\:. Then for each  $N \in \N$, there exists $g_N\in \mathscr{S}(\R^n)_{o}$ such that \begin{equation*}
\mathscr{F}g_N(\lambda)=\eta(\lambda/N)\:,\:\:\:\:\:\:\lambda\in [0,\infty)\:.
\end{equation*}
We note that
\begin{equation}\label{f1}
g_1(0)=\int_0^{\infty}\eta(\lambda)\lambda^{n-1}d\lambda=C_9>0\:,
\end{equation}
and for $s\in [0,\infty)$, by the Euclidean Fourier inversion applied to radial functions
\begin{eqnarray*}
g_N(s)&=& \int_0^{\infty}\eta\left(\frac{\lambda}{N}\right)\J_{\frac{n-2}{2}}(\lambda s)\lambda^{n-1}d\lambda\\
&=&N^n\int_0^{\infty}\eta(\lambda)\J_{\frac{n-2}{2}}(\lambda N s)\lambda^{n-1}d\lambda=N^ng_1(Ns)\:.
\end{eqnarray*}
Therefore, by (\ref{f1}) there exists $\varepsilon\in (0,1/2)$ such that
\begin{equation}\label{fn}
|g_N(s)|>\frac{C_9}{2}N^n\:,\:\:\:\:\:\:\text{for all, $0\le s<\frac{\varepsilon}{N}$.}
\end{equation}
Now from the definition of $g_N$, it follows that $Supp\:(\:\mathscr{F}g_N\:) \subset (N,2N)$. Then by Lemma \ref{schwartz_correspondence}, for each $N \in \N$, there exists $f_N \in \mathscr{S}^2(S)_o$ such that 
\begin{equation} \label{correspondence}
\lambda^{n-1} \mathscr{F}g_N(\lambda) = {|{\bf c}(\lambda)|}^{-2}\: \widehat{f_N}(\lambda)\:,
\end{equation}
for all $\lambda$\:. By (\ref{correspondence}), it follows that
\begin{eqnarray} \label{sharpness_eq6}
{\|f_N\|}_{H^\beta(S)} &=& {\left(\int_0^\infty {\left(\lambda^2 + \frac{Q^2}{4}\right)}^\beta \eta\left(\lambda/N\right)^2 \: \frac{\lambda^{2(n-1)}}{{|{\bf c}(\lambda)|}^{-4}}\:{|{\bf c}(\lambda)|}^{-2} d\lambda\right)}^{1/2} \nonumber\\
& \lesssim & N^{\beta +\frac{n}{2}} {\left(\int_1^2 \eta(\lambda)^2 \: \lambda^{2\beta + n-1}\:d\lambda\right)}^{1/2} \nonumber\\
& \lesssim & N^{\beta +\frac{n}{2}}\:.
\end{eqnarray}
Next, for $s\in[0,\varepsilon/N)$, using the spherical Fourier inversion along with the decomposition (\ref{ball_pf_eq1}), we write
\begin{eqnarray*}
S_{\psi,0}f_N(s) = f_N(s) &=&  \int_0^{\infty} \what{f_N}(\lambda)\varphi_\lambda(s)\:{|{\bf c}(\lambda)|}^{-2} d\lambda \\
&=&c_0\left(\frac{s^{n-1}}{A(s)}\right)^{1/2}\int_0^{\infty}\what{f_N}(\lambda)\J_{\frac{n-2}{2}}(\lambda s)\:{|{\bf c}(\lambda)|}^{-2} d\lambda \\
&&\:\:\:\:+ \int_0^{\infty}\what{f_N}(\lambda)E_1(\lambda, s)\:{|{\bf c}(\lambda)|}^{-2} d\lambda=f_{N,1}(s)+f_{N,2}(s)\:.
\end{eqnarray*}
Then using (\ref{correspondence}) and proceeding as in (\ref{linearized_max_fn_comparison}), it follows that
\begin{equation*}
|f_{N,1}(s)|\asymp |g_N(s)|\;,\:\:\:\:\:\text{for all, $s\in [0,\varepsilon/N)$,}
\end{equation*}
where the implicit constant is independent of $N$. Consequently, it follows from (\ref{fn}) that there exists a positive number $c$ (independent of $N$) such that
\begin{equation}\label{t1}
|f_{N,1}(s)|\geq cN^n\:,\:\:\:\:\:\text{for all, $s\in [0,\varepsilon/N)$.}
\end{equation}
We next show that the growth of $f_{N,2}$ is subordinated by $f_{N,1}$. As $s\in [0,\varepsilon/N)$ and $\varepsilon\in (0,1/2)$ it follows that for $\lambda\in (1,2)$ the quantity $\lambda Ns$ is smaller than $1$. Hence using the estimate (\ref{ball_pf_eq2}) of the error term  $E_1$, we get that for $s\in [0,\varepsilon/N)$,
\begin{eqnarray*}
|f_{N,2}(s)|&\le &\int_0^{\infty}\eta\left(\frac{\lambda}{N}\right)\:|E_1(\lambda, s)|\:{|{\bf c}(\lambda)|}^{-2} d\lambda\\
&=& N\int_1^2 \eta(\lambda)\:|E_1(\lambda N,s)|\:{|{\bf c}(\lambda N)|}^{-2}d\lambda\\
&\lesssim &N\|\eta\|_{\infty}\int_1^2s^2\:(\lambda N)^{n-1}d\lambda\\
&<& c'\varepsilon^2N^{n-2}\:.
\end{eqnarray*}
Therefore, using triangle inequality and (\ref{t1}), it follows that for all $s\in [0,\varepsilon/N)$ and $N$ sufficiently large,
\begin{eqnarray*}
|S_{\psi,0}f_N(s)|&\geq & |f_{N,1}(s)|-|f_{N,2}(s)|\\
&\geq & cN^n-c'\varepsilon^2N^{n-2}\geq c''N^n\:.
\end{eqnarray*}
This implies that for all $x \in B_{\varepsilon/N}$,
\begin{equation}  \label{sharpness_eq7}
S^*_{\psi}f_N(x) \ge c\:N^n\:,
\end{equation}
with the constant $c$ being independent of $N$. Now, if the estimate (\ref{estimates_on_balls}) holds, then by (\ref{sharpness_eq6}) and (\ref{sharpness_eq7}), it follows that
\begin{equation*}
{\left(\int_{0}^{\varepsilon/N} N^{np}\:s^{n-1}\:ds\right)}^{1/p} \lesssim N^{\beta +\frac{n}{2}}\:,
\end{equation*}
which implies that
\begin{equation*}
N^{n\left(1-\frac{1}{p}\right)} \lesssim N^{\beta +\frac{n}{2}}\:.
\end{equation*}
Then letting $N \to \infty$, it follows that
\begin{equation*}
n\left(1-\frac{1}{p}\right) \le \beta +\frac{n}{2}\:,
\end{equation*}
that is, $p \le 2n/(n-2\beta)$.

\medskip

{\bf Case 3}: {\em If $\beta=n/2$ then (\ref{estimates_on_balls}) does not hold for $p=\infty$.}

\medskip
\noindent
In this case we show the failure of the inequality
\begin{equation} \label{final_endpoint_eq1}
{\|S^*f\|}_{L^\infty(B)} \le C_B \: {\|f\|}_{H^{n/2}(S)}\:,
\end{equation}
for every ball centered at the identity $e$. We first note by the sufficient conditions of Theorem \ref{theorem} (and hence by Corollary \ref{cor}) that for all radial $f \in H^{n/2}(S)$, one has
\begin{equation} \label{final_endpoint_eq2}
\lim_{t\to 0}S_{\psi,t}f(x)=f(x)\:,\text{ for almost every } x \in S\:.
\end{equation} 
Now if we assume that the inequality (\ref{final_endpoint_eq1}) holds, combining it with (\ref{final_endpoint_eq2}), it would follow that
\begin{equation*}
{\|f\|}_{L^\infty(B)} = {\|\lim_{t\to 0}S_{\psi,t}f\|}_{L^\infty(B)} \le {\|S^*f\|}_{L^\infty(B)} \le C_B \: {\|f\|}_{H^{n/2}(S)}\:,
\end{equation*}
for every ball centered at the identity $e$. But this is not true \cite[Remark 5.1]{DR}. Hence we get the desired failure of (\ref{final_endpoint_eq1}). This completes the proof of Theorem \ref{theorem}. 

\begin{remark} \label{remark_on_sharpness_proofs}
It is curious to note that the oscillatory nature of the multiplier is crucially used in the proof of case 1, that is, to show that no maximal estimate (\ref{estimates_on_balls}) is possible if $\beta<1/4$. In the other two cases however, the crux of the arguments are weaved only in terms of the initial data and hence the multiplier plays no role there.
\end{remark}

\section{Concluding remarks}
In this section, we make some remarks and pose some new problems:
\subsection{Results for general dispersive equations:} Our proof of the sufficient conditions of Theorem \ref{theorem}, essentially works for the general class of  dispersive equations whose corresponding multipliers have phase functions $\psi$ satisfying the following asymptotic properties for some real numbers $\delta_1>0$ and $\delta_2>1$: 
\begin{equation*} 
\begin{cases}
	 |\psi'(\lambda)| \lesssim \lambda^{\delta_1 -1}\:,\text{  for } \lambda \in (0,1)  \:,\\
	 |\psi'(\lambda)| \lesssim \lambda^{\delta_2 -1}\:,\text{  for } \lambda \ge 1  \:,\\
	|\psi''(\lambda)| \asymp \lambda^{\delta_2 -2}\:,\text{  for } \lambda \ge 1  \:.
	\end{cases}
\end{equation*} 
Interestingly, in the proof of the necessity of the conditions of Theorem \ref{theorem} (along with the asymptotics of the first two derivatives), as clarified by the Transference principle (Theorem \ref{transference_principle}), it is the high frequency asymptotics of $\psi$ that turns out to be the heart of the matter. Thus dispensing of the explicit expressions of $\psi$, our arguments can be employed to obtain mapping properties (and consequently results on pointwise convergence) for general dispersive equations, even if one only has suitable asymptotic control on $\psi$ and its first and second derivatives.    

\subsection{Riemannian symmetric spaces of higher rank:} A special case of the class of Damek-Ricci spaces, is the class of rank one Riemannian symmetric spaces of non-compact type, that is, homogeneous spaces $G/K$, where $G$ is a connected, non-compact, semi-simple Lie Group with finite center and real rank one and $K$ is a maximal compact subgroup of $G$. In the case of rank one, the geometric notion of radiality coincides with the algebraic notion of $K$-biinvariance. It then becomes natural to pose the Carleson's problem for the Schr\"odinger equation (or more general dispersive equations) with $K$-biinvariant initial data on Riemannian symmetric spaces of non-compact type of arbitrary rank.

\section*{Acknowledgements} 
The author is supported by a research fellowship of Indian Statistical Institute.

\bibliographystyle{amsplain}

\begin{thebibliography}{amsplain}
\bibitem[ADY96]{ADY}  Anker, J-P., Damek, E. and Yacoub, C. {\em Spherical analysis on harmonic AN groups}. Ann. Scuola Norm. Sup. Pisa Cl. Sci. 4, 23 (1996), no. 4, 643-679.
\bibitem[AP14]{AP} Anker, J-P. and Pierfelice, V. {\em Wave and Klein-Gordon equations on hyperbolic spaces}. Anal. PDE, 2014, 7(4), pp.953-995.
\bibitem[APV15]{APV} Anker, J-P., Pierfelice, V. and Vallarino, M. {\em The wave equation on Damek-Ricci spaces}, Ann. Mat. Pura Appl. (4) 194 (2015), no. 3, 731-758.
\bibitem[As95]{A} Astengo, F. {\em A class of $L^p$ convolutors on harmonic extensions of H-type groups}. J. Lie Theory 5 (1995), no. 2, 147-164.
\bibitem[BD23]{Bhowmik} Bhowmik, M. and Dewan, U. {\em Spectral projections and resolvent estimates on Damek-Ricci spaces and their applications.} arXiv:2306.06875 .
\bibitem[Bo16]{Bourgain} Bourgain, J. {\em A note on the Schr\"odinger maximal function.} J. Anal. Math. 130 (2016), 393-396.
\bibitem[Ca80]{C} Carleson, L. {\em Some analytic problems related to statistical mechanics, Euclidean harmonic analysis.} Lecture Notes in Math. 779, Springer, Berlin, 1980, 5-45.
\bibitem[Co83]{Cowling} Cowling, M. {\em Pointwise behavior of solutions to Schr\"odinger equations, harmonic analysis.} Lecture Notes in Math. 992. Springer, Berlin, 1983, 83-90.
\bibitem[CDKR98]{CDKR} Cowling, M., Dooley, A., Kor\'anyi, A. and Ricci, F. {\em An approach to symmetric spaces of rank one via groups of Heisenberg type.} J. Geom. Anal. 8 (1998), no. 2, pp. 199–237.
\bibitem[DK82]{DK} Dahlberg, B.E.J. and Kenig, C.E. {\em A note on the almost everywhere behavior of solutions of the Schr\"odinger equation.} Lecture Notes in Math. 908. Springer-Verlag, Berlin, 1982, 205-208.

\bibitem[De2]{Dewan2} Dewan, U. {\em Pointwise convergence of solutions of the Schr\"odinger equation along general curves with radial initial data on Damek-Ricci spaces.} arXiv:2411.14020 .
\bibitem[De1]{Dewan} Dewan, U. {\em Regularity and pointwise convergence of solutions of the Schr\"odinger operator with radial initial data on Damek-Ricci spaces.} Annali di Matematica (2024). https://doi.org/10.1007/s10231-024-01523-2
\bibitem[DR24]{DR} Dewan, U. and Ray, S.K. {\em Mapping properties of the local Schr\"odinger maximal function with radial initial data on Damek-Ricci spaces.} arXiv:2411.04084 . 
\bibitem[DN17]{DN} Ding, Y. and Niu, Y. {\em Maximal estimate for solutions to a class of dispersive equation with radial initial data.} Front. Math. China 2017, 12(5): 1057-1084.
\bibitem[DGL17]{DGL} Du, X., Guth, L. and Li, X. {\em A sharp Schr\"odinger maximal estimate in $\R^2$.} Ann. Math. 2017, 186, 607-640.
\bibitem[DZ19]{DZ} Du, X. and Zhang, R. {\em Sharp $L^2$ estimates of the Schr\"odinger maximal function in higher dimensions.} Ann. Math. 2019, 189, 837-861.
\bibitem[FX24]{FX} Ferreira, L. and Xuan, P. {\em Dispersive estimates and generalized Boussinesq equation on hyperbolic spaces with rough initial data.} arXiv:2410.20472 \:.
\bibitem[GPW08]{GPW} Guo, Z., Peng, L. and Wang, B. {\em Decay estimates for a class of wave equations.} J. Funct. Anal. 254 (2008), pp. 1642-1660.
\bibitem[KS24]{Kumar} Kumar, P. and Sajjan, M. {\em Regularity of solution of the Schr\"odinger equation on Symmetric Space}. arXiv:2411.06104 .

\bibitem[Pr90]{Prestini} Prestini, E. {\em Radial functions and regularity of solutions to the Schr\"odinger equation.} Monatsh. Math. 109, 135-143 (1990).
\bibitem[RS09]{RS} Ray, S.K. and Sarkar, R.P. {\em Fourier and Radon transform on harmonic NA groups.} Trans. Amer. Math. Soc. 361 (2009), no. 8, 4269–4297.
\bibitem[Sj81]{SjolinOsc} Sj\"olin, P. {\em Convolution with oscillating kernels.} Indiana Univ Math J, 1981, 30: pp. 47-55.
\bibitem[Sj87]{Sjolin} Sj\"olin, P. {\em Regularity of solutions to the Schr\"odinger equation.} Duke Math J. 55(1987), 699-715.
\bibitem[Sj97]{Sjolin2} Sj\"olin, P. {\em $L^p$ maximal estimates for solutions to the Schr\"odinger equation.} Math. Scand., vol. 81, no. 1 (1997), pp. 35-68.
\bibitem[ST78]{ST} Stanton, R. J. and Tomas, P. A. {\em Expansions for spherical functions on noncompact symmetric spaces.} Acta Math. 140 (1978), no. 3-4, pp. 251-276.
\bibitem[St56]{Stein} Stein, E.M. {\em Interpolation of Linear Operators.} Trans. Amer. Math. Soc. vol. 83, no. 2 (1956), 482–492.
\bibitem[St86]{BigStein} Stein, E.M. {\em Oscillatory integrals in Fourier analysis. In: Stein E M, ed. Beijing Lectures in Harmonic Analysis.} Ann of Math Stud, vol 112. Princeton: Princeton Univ Press, 1986, pp. 307-355.
\bibitem[SW90]{SW} Stein, E.M. and Weiss, G. {\em Introduction to Fourier Analysis on Euclidean Spaces.} Princeton University Press, Princeton, New Jersey, Sixth printing, 1990.
\bibitem[Ve88]{Vega} Vega, L. {\em Schr\"odinger equations: pointwise convergence to the initial data.} Proc. Amer. Math. Soc. 102(1988), 874-878.
\bibitem[WZ19]{WZ} Wang, X. and Zhang, C. {\em Pointwise Convergence of Solutions to the Schr\"odinger Equation on Manifolds}. Canad. J. Math. Vol. 71(4), 2019, 983-995.


\end{thebibliography}

\end{document}